\documentclass[twoside]{IEEEtran}
\usepackage{amsmath,amstext,amssymb, txfonts}
\usepackage{mathrsfs}
\usepackage{graphicx,epstopdf}
\usepackage{verbatim}
\usepackage{ifthen}
\usepackage{multirow}
\usepackage{cite}
\usepackage[usenames,dvipsnames]{color}
\usepackage{xcolor}
\usepackage{lscape}
\usepackage{mdwlist}
\usepackage{url}
\usepackage{dcolumn}
\usepackage{amsfonts}
\usepackage{latexsym}
\usepackage{bm}
\usepackage{url}
\usepackage{subfigure}
\usepackage{xcolor}
\usepackage{framed} 
\usepackage[framed]{ntheorem}
\usepackage{empheq}

\usepackage{threeparttable}

\usepackage{algorithm}
\usepackage{algorithmicx}
\usepackage{algpseudocode}

\DeclareMathOperator{\diag}{diag}

\newboolean{showcomments}
\setboolean{showcomments}{true}
\newcommand{\masoud}[1]{  \ifthenelse{\boolean{showcomments}}
	{ \textcolor{red}{(Masoud says:  #1)}} {}  }
\newcommand{\slow}[1]{\ifthenelse{\boolean{showcomments}}
	{ \textcolor{red}{(Steven says:  #1)}}{}}

\def\ba{\begin{array}}
	\def\ea{\end{array}}
\newcommand{\beq}{\begin{align}}
	\newcommand{\eeq}{\end{align}}
\newcommand{\bq}{\begin{eqnarray}}
	\newcommand{\eq}{\end{eqnarray}}
\newcommand{\bqn}{\begin{eqnarray*}}
	\newcommand{\eqn}{\end{eqnarray*}}
\newcommand{\bee}{\begin{enumerate}}
	\newcommand{\eee}{\end{enumerate}}
\newcommand{\bi}{\begin{itemize}}
	\newcommand{\ei}{\end{itemize}}
\newcommand{\btab}{\begin{tabular}}
	\newcommand{\etab}{\end{tabular}}

\newtheorem{theorem}{Theorem}
\newtheorem{definition}{Definition}
\newtheorem{lemma}{Lemma}
\newtheorem{corollary}{Corollary}

\newtheorem{remark}{Remark}
\newtheorem{assumption}{Assumption}

\newcommand{\hL}{\mathcal{L}}
\newcommand{\hN}{\mathcal{N}}

\newcommand{\cD}{{\cal D}}

\newcommand{\cN}{{\cal N}}

\begin{document}

{\title{\LARGE Reverse and Forward Engineering of Local Voltage Control\\in Distribution Networks}
\author{Xinyang Zhou, Masoud Farivar, Zhiyuan Liu, Lijun Chen and Steven Low
		\thanks{X. Zhou is with the National Renewable Energy Laboratory, Golden, CO 80401, USA (Email: xinyang.zhou@nrel.gov).}\thanks{Z. Liu and L. Chen are with College of Engineering and Applied Science, University of Colorado, Boulder, CO 80309, USA (Emails: \{zhiyuan.liu, lijun.chen\}@colorado.edu).} \thanks{M. Farivar is with Google. (Email: mfarivar@gmail.com).}\thanks{S. Low is with the Division of Engineering and Applied Science, California Institute of Technology, Pasadena, CA 91125, USA (Email: slow@caltech.edu).}
	\thanks{Preliminary results of this paper have been presented at IEEE Conference on Decision and Control, Florence, Italy, 2013 \cite{farivar2013equilibrium}, IEEE International Conference on Smart Grid Communications, Miami, FL, 2015 \cite{farivar2015local}, and Annual Allerton Conference on Communication, Control, and Computing, Allerton, IL, 2015 \cite{zhou2015pseudo}. }
}
\maketitle
}
\begin{abstract}
The increasing penetration of renewable and distributed energy resources in distribution networks calls for real-time and distributed voltage control. In this paper we investigate local Volt/VAR control with a general class of control functions,  and show that the power system dynamics with non-incremental local voltage control can be seen as distributed algorithm for solving a well-defined optimization problem (reverse engineering). The reverse engineering further reveals a fundamental limitation of the non-incremental voltage control: the convergence condition is restrictive and prevents better voltage regulation at equilibrium. This motivates us to design two incremental local voltage control schemes based on the subgradient and pseudo-gradient algorithms respectively for solving the same optimization problem (forward engineering). The new control schemes decouple the dynamical property from the  equilibrium property, and have much less restrictive convergence conditions. 
This work presents another step towards developing a new foundation -- network dynamics as optimization algorithms -- for distributed realtime control and optimization of future power networks.

\end{abstract}


\section{Introduction}\label{sect:intro}
Traditionally, given  the predictable and relatively slow changes in power demand, capacitor banks and load tap changers are switched a few times per day to regulate the voltage in distribution systems; see, e.g.,~\cite{baran1989optimala, baran1989optimalb}. However, with the increasing penetration of renewable energy resources such as photovoltaic (PV) and wind turbines in both residential and commercial settings \cite{mccrone2017global, ren21}, the intermittent and fast-changing renewable energy supply introduces rapid fluctuations in voltage that are beyond the capability of those traditional voltage regulation schemes and thus calls for new voltage control paradigms. 

\subsection{Inverter-Based Voltage Regulation}

Even though the current IEEE Standard 1547 \cite{standard1547a} requires distributed generation to operate at unity power factor, inverters can readily adjust real and reactive power outputs to stabilize voltages and cope with fast time-varying conditions. Indeed, the IEEE Standards group is actively exploring a new inverter-based Volt/Var control. Unlike the capacity banks or tap changers, inverters can push and pull reactive power much faster, in a much finer granularity and with low operation cost, enabling real-time distributed control that is needed for the future power grid with a large number of renewable and distributed energy resources. 

Inverter-based voltage regulation has been studied extensively in literature. Related work roughly falls into the following categories:
\begin{enumerate}
\item Centralized control: By collecting all the required information and computing a global optimal power flow (OPF) problem, a central controller determines optimal set-points for the whole system; see, e.g., \cite{kekatos2015stochastic,farivar2011inverter,lavaei2012zero,low2014convexa,low2014convexb}. Centralized control can incorporate general objectives and operational constraints, but suffers from considerable communication overhead and long computation time especially when the size of the system is large. So, it usually cannot provide fast control. 

\item Distributed control: For the OPF problems of certain structures,  one can design algorithms to distribute the computation with coordinating communication, which is conducted either between a central controller and agents in a hierarchical way, e.g., \cite{kekatos2016voltage, baker2017network, dall2014optimal, zhou2017incentive, zhou2017discrete, li2014real, bolognani2013distributed, magnusson2017voltage}, or among neighborhoods of individual agents without any central controller, e.g., \cite{peng2016distributed, vsulc2014optimal, shi2015distributed, magnusson2015distributed, bazrafshan2017decentralized, wu2017distributed}. Given the required communication infrastructure, the scheme based on distributed OPF algorithm can provide scalable voltage control. 

\item Local control: Based on only local information, local voltage control provides fast response and, without the need of communication, allows simple and scalable implementation; see, e.g., \cite{farivar2013equilibrium,simpson2017voltage, zhu2016fast,zhou2016local, zhou2016incremental, jahangiri2013distributed, robbins2013two, turitsyn2011options}. 
\end{enumerate}

In this paper, we focus on analysis and design of local voltage control. Characterization of local control, especially systemwide properties arising from the interaction between local controls, is challenging.  In literature, the work such as \cite{turitsyn2011options, andren2015stability} lack analytical characterization. Other work such as \cite{jahangiri2013distributed, robbins2013two} provide stability analysis but lack systemwide performance characterization. There are work that provide rigorous performance analysis for stability and systemwide performance, but are subject to control functions of particular type, e.g., linear control functions without deadband \cite{zhu2016fast,zhang2013local} and quadratic control functions  \cite{simpson2017voltage}.




\subsection{Reverse and Forward Engineering}

Different from other work, in this paper we consider local voltage control with general monotone control functions, and seek a principled way to guide systematic analysis and design of local voltage control with global perspective through the approach of reverse and forward engineering. We first develop models to understand the systemwide properties arising from the interaction between local controls, in particular, whether the power system dynamics with the existing controls can be interpreted as distributed algorithms for solving certain optimization problems, i.e., {\em network dynamics as optimization algorithms}. We then leverage the insights obtained form the reverse engineering to design new local volatge control schemes according to distributed algorithms for solving the resulting optimization problem  (or its variant that incorporates new design objective and/or constraints). 

Specifically, we first lay out a general framework for reverse engineering power system dynamics with non-incremental local voltage control with general control functions, and show that it can be seen as a distributed algorithm for solving a well-defined optimization problem. We characterize the condition under which the dynamical system converges, which is however very restrictive and prevents better voltage regulation at the equilibrium  (or optimum): aggressive control functions are preferred for better voltage regulation at equilibrium, while less aggressive ones are preferred for convergence. We are therefore motivated to find a way to decouple the dynamical property from the equilibrium property. 


Notice that the optimization-based model does not only provide a way to characterize the equilibrium and establish the convergence of power system dynamics with local control (i.e., reverse engineering), but also suggests a principled way to engineer the control  to achieve the desired property (i.e., forward engineering). In particular, new control schemes with better dynamical properties can be designed based on different optimization algorithms for solving the same optimization problem. Accordingly, we propose an incremental local voltage control scheme based on the (sub)gradient algorithm for solving the same optimization problem. This new control scheme decouples the equilibrium property and the convergence property, and has much less restrictive convergence condition. However, it converges to only within a small neighborhood of the equilibrium. Furthermore, it requires computing the inverse of the control function, which may incur considerable computation overhead. We thus propose another incremental local voltage control scheme based on a pseudo-gradient algorithm that has better convergence property and simpler implementation than the (sub)gradient control while achieving the same equilibrium. 


Similar idea of reverse and forward engineering based on the perspective of network dynamics as optimization algorithms has been applied to distributed real-time frequency control of the power system, e.g., \cite{chen2016reverse,li2016connecting,Zhao-2012-LC-SGC,Zhang-2013-ACC,Zhao-2014-TAC,Mallada-2014-IFAC,Dorfler-2014-CDC}, as well as synchronization of the network of coupled oscillators \cite{zhou2015new}. This paper presents another step towards developing a new foundation -- network dynamics as optimization algorithms -- for distributed realtime control and optimization of future power networks.

The rest of the paper is organized as follows. Section~\ref{sec:model} describes the system model and introduces the non-incremental local voltage control. Section~\ref{sec:eq} investigates the equilibrium and dynamical properties of the non-incremental local control by reverse engineering. Section~\ref{sec:forward1} proposes two incremental local voltage control schemes that decouple the equilibrium and convergence properties and have much less restrictive convergence conditions. Section~\ref{sec:numerical} provides numerical examples to complement the theoretical analysis, and Section~\ref{sec:conclusion} concludes the paper.


\section{Network Model and Local Voltage Control}\label{sec:model}
Consider a tree graph $\mathcal{G}=\{\hN \cup\{0\}, \hL\}$ that represents 
a radial distribution network consisting of $n+1$ buses and a set $\hL$ of undirected lines between these buses. Bus 0 is the substation bus (slack bus) and is  assumed to have a fixed voltage of $v_0=1$~p.u.  Let $\hN:=\{1, \ldots, n\}$. Due to the tree topology, we also have the cardinality of the line set $|\hL|=n$. 
For each bus $i\in \hN$, denote by $\hL_i \subseteq \hL$ the set of lines on the unique path from bus $0$ to bus $i$, $p_i^c$ and $p_i^g$ the real power consumption and generation respectively, and  $q_i^c$ and $q_i^g$  the reactive power consumption and generation respectively. Let $v_i$ be the magnitude of the complex voltage (phasor) at bus $i$. For each line $(i, j)\in \hL$, denote by $r_{ij}$ and $x_{ij}$ its resistance and reactance, and  $P_{ij}$ and $Q_{ij}$ the real and reactive power from bus $i$ to bus $j$. Let $\ell_{ij}$ denote the squared magnitude of the complex branch current (phasor) from bus $i$ to bus $j$. We summarize some of the notations used in this paper in Section~\ref{sec:not}.

\subsection{Notation}\label{sec:not}

        \begin{tabular}{ll}   
        $\hN$ & set of buses excluding bus $0$, $\hN:=\{1, ..., n\}$\\
        $\hL$ & set of power lines\\
        $\hL_i$ & set of lines from bus 0 to bus i \\ 
        $p_i^c, q_i^c$ & real, reactive power consumption at bus $i$\\
        $q_i^g, q_i^g$& real, reactive power generation at bus $i$\\
        $P_{ij}, Q_{ij}$ & real and reactive power flow from $i$ to $j$\\
        $r_{ij}, x_{ij}$ & resistance and reactance of line $(i,j)$\\
        $v_i$  &    magnitude of complex voltage at bus $i$\\
        $\ell_{ij}$ & squared magnitude of complex current of\\
        & line $(i,j)$\\
        $\Omega_i$& feasible power set of inverter $i$; $\Omega := \bigtimes_{\ i=1}^{\ n} \!\Omega_i$\\
        $[\,]_{\Omega_i}$ & projection onto set $\Omega_i$\\
        $\sigma_{max}(\,)$ & maximum singular value of a matrix\\
        $\lambda_{max}(\,)$ & maximum eigenvalue of a matrix\\
        $\overline{\alpha}_i$ & upper-bound of the (sub)derivative $df_i(v_i)/dv_i$; \\
        &$\overline{A}:=\diag\{\overline{\alpha}_1,\ldots,\overline{\alpha}_n\}\in\mathbb{S}_{++}^{N}$\\
        \end{tabular}
        
 \vspace{3mm}
 A quantity without subscript is usually a vector with appropriate components defined earlier, 
e.g., $v := (v_i, i\in\hN), q^g := (q_i^g, i\in \hN)$.

\subsection{Linearized Branch Flow Model} 



We adopt the following branch flow model  introduced in \cite{baran1989optimala, baran1989optimalb} ({\it DistFlow equations}) to model a \emph{radial} distribution system:
\begin{subequations}\label{eq:bfm}
	\begin{eqnarray}
		P_{ij} &=& p_j^c - p_j^g + \sum_{k: (j,k)\in \hL} P_{jk}+  r_{ij}  \ell_{ij}  \label{p_balance}, \\
		Q_{ij} &=&  q_j^c-q_j^g + \sum_{k: (j,k)\in \hL} Q_{jk} + x_{ij} \ell_{ij} \label{q_balance},\\
		v_j^2 &=&  v_i^2 - 2 \Big(r_{ij} P_{ij} + x_{ij} Q_{ij} \Big) + \Big(r_{ij}^2+x_{ij}^2\Big) \ell_{ij} \label{v_drop},\\
		\ell_{ij}v_i^2 &=&   P_{ij}^2 + Q_{ij}^2  \label{currents}.
	\end{eqnarray}
\end{subequations}
Following \cite{baran1989network} we assume that the terms involving $\ell_{ij}$ are zero for all $(i,j) \in \hL$ in \eqref{eq:bfm}.
This approximation neglects the higher order real and reactive power loss terms.
Since losses are typically much smaller than power flows $P_{ij}$ and $Q_{ij}$, it only
introduces  a small relative error, typically on the order of $1\%$. 
We further assume that $v_i \approx 1,\ \forall i$ so that we can set $v_j^2 - v_i^2 = 2 (v_j - v_i)$ in equation \eqref{v_drop}. This approximation introduces a small relative error of at most $0.25\%$ if there is a $5\%$ deviation in voltage magnitude.

With the above approximations \eqref{eq:bfm} is simplified to the following linear model:
\begin{eqnarray*}
	P_{ij} &=& \sum_{k\in \beta(j) }\left(p_k^c - p_k^g\right)  \label{p_balance_lin}, \\
	Q_{ij} &=&  \sum_{k\in \beta(j) }\left( q_j^c-q_j^g \right) \label{q_balance_lin}, \\ 
	v_i - v_j & = &  r_{ij} P_{ij} +  x_{ij} Q_{ij}    \label{v_drop_lin},
\end{eqnarray*}
where $\beta (j)$ is the set of all descendants of  bus $j$ including bus $j$ itself, i.e., $\beta(j)=\left\{ i | \hL_j \subseteq \hL_i\right\}$. 
This yields an explicit solution for $v_i$ in terms of $v_0$ (which is given and fixed):
\begin{eqnarray*}\label{Voltage_explicit}
	&&v_0  - v_i~=   \sum\limits_{(j,k)\in \hL_i} {r_{jk} P_{jk}}  + \sum\limits_{(j,k)\in \hL_i} {x_{jk} Q_{jk}}\nonumber \\
	&=&  \sum\limits_{(j,k)\in \hL_i} {r_{jk} \left( \sum_{h\in \beta(k) }\left(p_h^c - p_h^g\right) \right)}  + \sum\limits_{(j,k)\in \hL_i} {x_{jk} \left( \sum_{h\in \beta(k) }\left( q_h^c-q_h^g \right)  \right)} \nonumber \\
	&=&  \sum_{j\in \hN}\left(p_j^c - p_j^g\right) { \left( \sum_{(h,k)\in \hL_i \cap \hL_j} \!\!\!\!r_{hk} \right)}  + \sum_{j\in \hN}\left(q_j^c - q_j^g\right) { \left( \sum_{(h,k)\in \hL_i \cap \hL_j} \!\!\!\!x_{hk} \right)} \nonumber \\
	&=&  \sum_{j\in \hN}R_{ij}\left(p_j^c - p_j^g\right)   + \sum_{j\in \hN} X_{ij}\left(q_j^c - q_j^g\right),
\end{eqnarray*}
where 
\begin{eqnarray}
	R_{ij}:= \!\!\! \sum_{(h,k)\in \hL_i \cap \hL_j}\!\!\!\! r_{hk}, \ \ \ \ X_{ij}:=\!\!\!\! \sum_{(h,k)\in \hL_i \cap \hL_j}\!\!\!\! x_{hk}  \label{X_def}.
\end{eqnarray}
\begin{figure}[htbp]
	\centering
	\includegraphics[width=.32\textwidth]{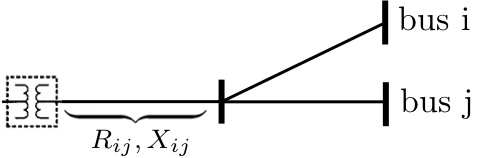}
	\caption{$\hL_i \cap \hL_j$ for  two arbitrary buses $i, j$ in the  network and the corresponding mutual voltage-to-power-injection sensitivity factors $R_{ij}, X_{ij}$.}
	\label{fig:RX}
\end{figure}
\noindent

Fig.~\ref{fig:RX} gives an illustration of  $\hL_i \cap \hL_j$ for two arbitrary buses $i$ and $j$ in a radial network and the corresponding $R_{ij}$ and $X_{ij}$. Since
	\begin{subequations}
	\begin{eqnarray}
	R_{ij} & = & \frac{\partial v_i}{\partial p^g_j}  \ = \  - \frac{\partial v_i}{\partial p^c_j}, \\
	X_{ij} & = & \frac{\partial v_i}{\partial q^g_j} \ = \ - \frac{\partial v_i}{\partial q^c_j},
	\end{eqnarray}
	\end{subequations}
$R_{ij}$, $X_{ij}$ are also referred to as the mutual voltage-to-power-injection sensitivity factors.

Define a resistance matrix $R=[R_{ij}]_{n\times n}$ and a reactance matrix $X=[X_{ij}]_{n\times n}$. 
Both matrices are symmetric.
With the matrices $R$ and $X$ the linearized branch flow model can be summarized compactly as:
\begin{eqnarray}  
	v &=& \overline{v}_0 + R (p^g - p^c) + X(q^g - q^c),
\end{eqnarray}
where $\overline{v}_0 = [v_0, \dots, v_0]^{\top}$ is an $n$-dimensional vector. In this paper we assume that $\overline{v}_0, p^g, p^c, q^c$ are given constants.  The only variables are (column) vectors $v := [v_1, \dots, v_n]^{\top}$ of squared voltage magnitudes
and $q^g := [q^g_1, \dots, q^g_n]$ of generated reactive powers. Let  $\tilde{v} = \overline{v}_0 + R (p^g - p^c) - X q^c$,  which is a constant vector. For notational simplicity, we will henceforth ignore the superscript in $q^g$ and write $q$ instead. Then the linearized branch flow model reduces to the following simple form:
\begin{eqnarray}  \label{model_2}
	v &=& Xq + \tilde{v}.
\end{eqnarray}

We have the following result. 
\begin{lemma}\label{lemma:X}
	The matrices $R$ and $X$ are positive definite.\hfill$\Box$
\end{lemma}
	\begin{proof} The proof uses the fact that the resistance and reactance values of power lines in the network are all positive. Here we give a proof for the reactance matrix $X$, and exactly the same argument applies to the resistance matrix $R$.   
	
	We prove by induction on the number $k$ of buses in the network, excluding bus 0 (the root bus). The base case of  $k=1$ corresponds to a two-bus network with one line. Here $X$ is obviously a positive scalar that is equal to the reactance of the line connecting the two buses. 
	
	Suppose that the theorem holds for all $k\leq n$. For the case of $k=n+1$ we consider two possible network topologies as shown in Figure \ref{fig:PSD}:
	\begin{figure}[htbp]	
		\centering
		\subfigure[Case1:  degree of bus $0$ is greater than 1]{
			\label{subfig:template_signalerr_125}
			\includegraphics[width=.18\textwidth]{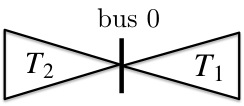}
		}
		\hspace{0.2in}
		\subfigure[Case 2:  degree of  bus 0 is 1]{
			\centering
			\label{subfig:template_signal_125}
			\includegraphics[width=.18\textwidth]{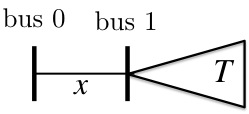}
		}
		\caption{Two possible network structures} 
		\label{fig:PSD}
	\end{figure}
	
	{\em Case 1:  bus $0$ is of degree greater than 1.} 
	Split the network into two different trees rooted at bus $0$, denoted by $T_1$ and $T_2$, each of which has no more than $n$ buses excluding bus $0$. Denote by $X_1$ and $X_2$ respectively the reactance matrices of $\mathcal{T}_1$ and $\mathcal{T}_2$. By induction assumption $X_1$ and $X_2$ are positive definite. Note that the set $\hL_i$ of lines on the unique path from bus $0$ to bus $i$  must completely lie inside either $\mathcal{T}_1$ or $\mathcal{T}_2$, for all $i$. Therefore, by definition (\ref{X_def}), the reactance matrix $X$ of the network has the following block-diagonal  form:
	\begin{equation}  \nonumber
	X_{ij}=\left\{ \begin{array}{rcl}
	X_{1_{ij}},  ~~ i,j\in \mathcal{T}_1~~\\
	X_{2_{ij}}, ~~i,j\in \mathcal{T}_2~~\\
	0, ~~~~\mbox{otherwise}  \\ 
	\end{array}\right.
	~ \Rightarrow~ X=\left[\begin{array}{cc}
	{{X_1}} & 0\\
	0 & {{X_2}}
	\end{array}\right].
	\end{equation}
	Since $X_1$ and $X_2$ are positive definite, so is $X$. 
	
	{\em Case 2: bus $0$ is of degree 1.} 
	Suppose without loss of generality that bus 0 is connected to bus $1$. 
	Denote by $x$ the reactance of the line connecting buses $0$ and $1$, and $\mathcal{T}$ the tree rooted at bus $1$, excluding bus $0$.
	Denote by $Y$ the reactance matrix of $\mathcal{T}$, and by induction assumption, $Y$ is positive definite. Note that, for all buses $i$ in the network, the set $\hL_i$ includes the single line that connects buses $0$ and $1$.  Therefore, by definition (\ref{X_def}), the reactance matrix X has the following form:
	\begin{equation}  \nonumber
	X_{ij}=\left\{ \begin{array}{rccl}
	Y_{ij}+x, ~~~ i, j\in \mathcal{T}\\
	x, ~~~~~~ \mbox{otherwise}  \\ 
	\end{array}\right.
	\Rightarrow~ X=\left[\begin{array}{cc}
	x  & \ldots x\\
	\vdots &  \vdots\\
	x & \ldots x
	\end{array}\right] + \left[\begin{array}{cc}{{0}} & 0\\
	0 & {{Y}}
	\end{array}\right],
	\end{equation}
	One can verify that, when $Y$ is positive definite and $x$ is positive,  $X$ is positive definite. This concludes the proof.
\end{proof}

We also refer to \cite{ zhu2016fast} for an alternative proof of the same result.

\subsection{Inverter Model}
 
At each bus $i$ there is an inverter that can generate non-negative real power $p_i$ and reactive power $q_i$ that can have either sign. $p_i$ and $q_i$ are constrained by the apparent power capability $s_i$ of the inverter as follows:
 \begin{eqnarray}
 0\leq p_i\leq s_i,\ \  0\leq |q_i|\leq s_i,\ \  p_i^{2}+q_i^{2}\leq s_i^2.\label{eq:apparent}
 \end{eqnarray}
 Consider power ratio $\cos\rho_i$ with $0\leq \rho_i \leq \pi/2$ such that
 \begin{eqnarray}
 p_i/s_i \geq  \cos\rho_i.\label{eq:powerratio}
 \end{eqnarray}
Given non-controllable $p_i\leq s_i$, the feasible (reactive) power set $\Omega_i$ for inverter $i$ is given by:
 \begin{eqnarray}
 \Omega_i &:=& \left\{q_i \ \big| \ {q_i}^{\text{min}}  \leq q_i \leq {q_i}^{\text{max}}\right\},
 \end{eqnarray}
 where, based on (\ref{eq:apparent})--(\ref{eq:powerratio}), 
 \begin{eqnarray*}
 q_i^{\max}&=&\min\left\{p_i\tan\rho_i,\sqrt{s_i^2-p_{i}^{2}} \right\},\\
 q_i^{\min}&=&\max\left\{-p_i\tan\rho_i,-\sqrt{s_i^2-p_{i}^{2}} \right\}. 
 \end{eqnarray*}
Here, $p_{i}$ is further assumed to be sized appropriately to provide enough freedom in $q_i$ \cite{turitsyn2011options}.
For buses without controllable inverters, one can set $q_i=q_i^{\max}=q_i^{\min}$ and $\Omega_i$ becomes a singleton. Define $\Omega := \bigtimes_{\ i=1}^{\ n} \Omega_i$ for notational simplicity.

%
%
%
%

\subsection{Local Volt/VAR Control}

The goal of Volt/VAR control in a distribution network is to maintain the bus voltages $v$ to within a tight range around their nominal values $v^{\text{nom}}_i=1~\text{p.u.}, ~i\in\hN$ by provisioning reactive power injections $q$. 
This can be modeled as a feedback dynamical system with state $\big(v(t), q(t)\big)$ at discrete time $t$.   A general Volt/VAR control algorithm maps the current state $\big(v(t), q(t)\big)$ to a new reactive power injections $q(t+1)$.  The new $q(t+1)$ updates voltage magnitudes $v(t+1)$ according to \eqref{model_2}. Usually $q(t + 1)$ is determined either completely or partly by a certain Volt/VAR control function defined as follows:

\begin{definition}
	A Volt/VAR control function $f : \mathbb{R}^{n} \rightarrow  \mathbb{R}^{n} $ is a collection of	local control functions $f_i: \mathbb{R}\rightarrow\mathbb{R}$, each of which maps the current local voltage $v_i$  to a local control variable $u_i$ in reactive power at bus $i$:
	\begin{equation} \label{VVC_function}
		u_i  \ = \ f_i (v_i-v_i^{\text{nom}}).\quad \forall i\in \hN.
	\end{equation}
\end{definition}

The control functions $f_i$ are usually decreasing but not always strictly decreasing because of a potential deadband where the control signal $u_i$ is set to zero to prevent too frequent actuation.
We assume for each bus $i\in\mathcal{N}$ a symmetric deadband $(v_i^{\text{nom}}-\delta_i/2, v_i^{\text{nom}}+\delta_i/2)$ with $\delta_i\geq 0$ around the nominal voltage $v_i^{\text{nom}}$. The following assumptions are made for $f_i$.
\begin{assumption}\label{A1}
	 The control functions $f_i$ are non-increasing in $\mathbb{R}$ and strictly decreasing and differentiable in $(-\infty, -\delta_i/2)\cup (\delta_i/2, +\infty)$. 
\end{assumption}
\begin{assumption}\label{A2}
	The derivative of the control function $f_i$ is upper-bounded, i.e., there exist $\overline{\alpha}_i>0$ such that $|f_i^{\prime}(v_i)|\leq\overline{\alpha}_i$ for all feasible $v_i\in(-\infty, -\delta_i/2)\cup (-\delta_i/2,\delta_i/2) \cup (\delta_i/2, +\infty),~\forall i\in\hN$.
\end{assumption}

Assumption~\ref{A2} ensures that an infinitesimal change in voltage does not lead to a jump in the control variable. Define $\overline{A}:=\diag\{\overline{\alpha}_1,\ldots,\overline{\alpha}_n\}\in\mathbb{S}_{++}^{N}$, and let $M=\sigma_{\max}(\overline{A}X)$ denote the largest singular value of $\overline{A}X$. We have the following result. 

\begin{lemma}[Lipschitz continuity]\label{lemma:lipschitz}
Suppose Assumptions~\ref{A1}--\ref{A2} hold. For any $q,q'\in\Omega$, we have 
\begin{eqnarray}
\|f(v(q)-v^{\text{nom}})- f(v(q')-v^{\text{nom}})\|_2\leq M \|q-q'\|_2.\label{eq:contM}
\end{eqnarray}
\end{lemma}
\begin{proof}
Without loss of generality, assume that $v_i(q)\geq v_i(q')$. If both $v_i(q)$ and $v_i(q')$ are in $(-\infty, v_i^{\text{nom}}-\delta_i/2]$ or in $[v_i^{\text{nom}}+\delta_i/2, +\infty)$, by the mean value theorem we have $|f_i(v_i(q)-v_i^{\text{nom}})- f_i(v_i(q')-v_i^{\text{nom}})| \leq \bar{\alpha}_i |v_i(q)- v_i(q')|$. If both are in $[v_i^{\text{nom}}-\delta_i/2, v_i^{\text{nom}}+\delta_i/2]$, $0=|f_i(v_i(q)-v_i^{\text{nom}})- f_i(v_i(q')-v_i^{\text{nom}})| \leq \bar{\alpha}_i |v_i(q)- v_i(q')|$. If $v_i(q)\in [v_i^{\text{nom}}+\delta_i/2, +\infty)$ and $v_i(q')\in [v_i^{\text{nom}}-\delta_i/2, v_i^{\text{nom}}+\delta_i/2]$, $|f_i(v_i(q)-v_i^{\text{nom}})- f_i(v_i(q')-v_i^{\text{nom}})|= |f_i(v_i(q)-v_i^{\text{nom}})- f_i(\delta_i/2))|\leq \bar{\alpha}_i |v_i(q)- (v_i^{\text{nom}}+\delta_i/2)| \leq \bar{\alpha}_i |v_i(q)- v_i(q')|$, where the first inequality follows from the mean value theorem. Similarly, we can show that $|f_i(v_i(q)-v_i^{\text{nom}})- f_i(v_i(q')-v_i^{\text{nom}})| \leq \bar{\alpha}_i |v_i(q)- v_i(q')|$ holds under other situations too. Therefore,
\begin{eqnarray}
\nonumber \|f(v(q)-v^{\text{nom}})- f(v(q')-v^{\text{nom}})\|_2 \leq \|\overline{A} (v(q)-v(q'))\|_2, 
\end{eqnarray}
from which we have
\begin{eqnarray}
\nonumber \|f(v(q)-v^{\text{nom}})- f(v(q')-v^{\text{nom}})\|_2 &\leq& \|\overline{A} X (q-q')\|_2\\
\nonumber &\leq & M \|q-q'\|_2. 
\end{eqnarray}
\end{proof}


See Fig.~\ref{controlfunction} (left) for an illustrative example of a piecewise linear droop control function based on IEEE Standard 1547 \cite{standard1547a}:
\begin{eqnarray} \label{eq:plf}
f_i(v_i) = -\alpha_i \left[ v_i  - {\delta_i}/{2}\right]^{+} + \alpha_i \left[ - v_i - {\delta_i}/{2} \right]^{+}
\end{eqnarray}
with slope $-\alpha_i$ in $(-\infty, -\delta_i/2)$ and $(\delta_i/2, +\infty)$. Notice that our design and analysis in this paper are not limited to the linear control functions.
\begin{figure*}
	\centering
	\includegraphics[trim = 0mm 0mm 0mm 0mm, clip, scale=0.54]{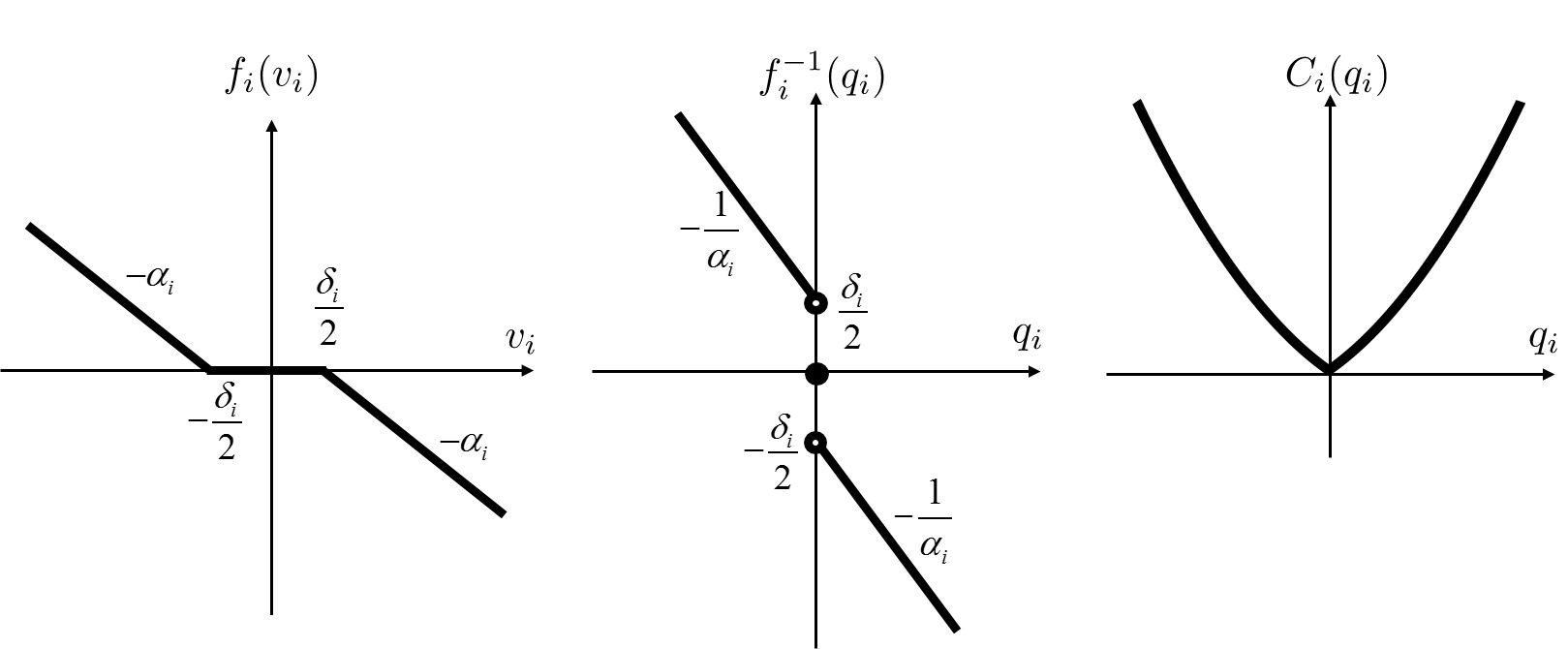}
	\caption{(left) The piecewise linear control function (\ref{eq:plf}), (middle) its inverse function~(\ref{eq:inverse}), and (right) the corresponding cost function~(\ref{eq:integral}). } \label{controlfunction}
\end{figure*}

Motivated by the IEEE Standard 1547, we consider a ``non-incremental'' control where the reactive power $q_i  = u_i,~i\in\hN$, and obtain the following 
dynamical system $\cD_1$ for the local Volt/VAR control:
\vspace{-1mm}
\begin{subequations}
	\begin{empheq}[left=\empheqlbrace]{align}
	\hspace{2mm}v(t)\hspace{2mm}&=\hspace{2mm}Xq(t)+\tilde{v}\label{dyn:D3a}\\
	\hspace{2mm}q_i(t+1)\hspace{2mm}&=\hspace{2mm}\big[f_i\big(v_i(t)-v_i^{\text{nom}}\big)\big]_{\Omega_i},~i\in\hN, \label{dyn:D1b}
	\end{empheq} 
\end{subequations}
where $[~]_{\Omega_i}$ denotes the projection onto the set $\Omega_i$.
A fixed point $\left({v}^*, q^* \right)$ of the above dynamical system, defined as follows, represents an equilibrium operating point of the network.
\begin{definition} 
	$\left({v}^*, q^* \right)$ is called an {\em equilibrium point} of $\cD_1$,
	if it satisfies 
	\begin{subequations}
	\begin{eqnarray}
	v^* & = &  Xq^* + \tilde{v},\\
	q^* & = & \big[f(v^*-v^{nom})\big]_{\Omega}.
	\end{eqnarray}
	\end{subequations}
\end{definition}


In the next section, we characterize the equilibrium  and dynamical properties of the system $\cD_1$ by showing that it is an distributed algorithm for solving a well-defined optimization problem. 

\section{Reverse Engineering}  \label{sec:eq}



Since $f_i$ is non-increasing, a (generalized) inverse $f_i^{-1}$ exists.
In particular, at the origin, we assign $f_i^{-1}(0) = 0$ corresponding to the deadband $[ - \delta_i/2,  + \delta_i/2]$ of $f_i$. This may introduce discontinuity to $f_i^{-1}$ at $q_i = 0$ if the deadband $\delta_i> 0$, i.e.,
\begin{eqnarray}
f_i^{-1}(0^+)\leq-\delta_i/2~\text{and}~f_i^{-1}(0^-)\geq\delta_i/2,\label{eq:discontinue}
\end{eqnarray}
where $0^+$ and $0^-$ represent approaching 0 from right and left, respectively.

Define a cost function for provisioning reactive power at each bus $i\in\hN$ as:
\begin{eqnarray}
C_i(q_i)  & := & - \int_{0}^{q_i} f_i^{-1}(q) \, dq,\label{eq:cost}
\end{eqnarray}
which is convex since $f_i^{-1}$ is decreasing.  Then, given $v_i(t)$,  $q_i(t+1)$ in \eqref{dyn:D1b} is the unique solution to the following optimization problem:
\bq
q_i({t+1}) & = & \underset{q_i\in\Omega_i}{\arg\min}
\ \  C_i(q_i) + q_i \left( v_i(t) - v_i^{\text{nom}} \right),
\label{eq:dynamic2b}
\eq
i.e., \eqref{dyn:D1b} and \eqref{eq:dynamic2b} are equivalent specification of
$q_i(t+1)$. 

Take for example  the piecewise linear control function (\ref{eq:plf}).  Its inverse is given by:
\bq
f_i^{-1}(q_i) & := & \left\{\begin{array}{lcl}
			 - \frac{q_i}{\alpha_i} + \frac{\delta_i}{2} 
						& & \text{if } q_i < 0,    \\
			0  & & \text{if } q_i = 0,    \\
			 - \frac{q_i}{\alpha_i} - \frac{\delta_i}{2} 
						& & \text{if } q_i > 0,						
			\end{array}  \right. \label{eq:inverse}
\eq

\noindent
and the corresponding cost function is given by:
\bq\label{eq:scf}
C_i(q_i) & = & \left\{ \begin{array}{lcl}
		\frac{1}{2\alpha_i}q_i^2 - \frac{\delta_i}{2} q_i  & \text{if} & q_i 
		\leq 0 ,
		\\
		\frac{1}{2\alpha_i}q_i^2 + \frac{\delta_i}{2} q_i & \text{if} & q_i \geq 0.
		\end{array}  \right. \label{eq:integral}
\eq
See Fig.~\ref{controlfunction} (middle and right) for illustration. 


\subsection{Equilibrium}

Consider the function $F(q): \Omega \rightarrow \mathbb R$:
\begin{eqnarray}
F(q) & := & C(q) + \frac{1}{2}q^{\top} X q + q^{\top} \Delta \tilde{v},\label{eq:obj}
\end{eqnarray}
where $C(q)=\sum_{i\in \hN} C_i(q_i)$ and $\Delta \tilde{v} := \tilde{v} - v^{\text{nom}}$, and a global optimization problem:
\begin{eqnarray}
\min_{q\in \Omega} ~~ F(q).
\label{eq:defminF}
\end{eqnarray}

\begin{theorem}\label{thm:eq}
Suppose Assumption~\ref{A1} holds.  Then $\cD_1$ has a unique equilibrium point. Moreover, a point $(v^*, q^*)$ is an equilibrium of $\cD_1$ if and only if $q^*$ is the unique optimal solution of \eqref{eq:defminF} and $v^* = Xq^* + \tilde{v}$.
\end{theorem}

%

\begin{proof}
By Lemma \ref{lemma:X} the matrix $X$ is positive definite. This implies that the objective function $F(q)$ is \emph{strongly} convex. Hence, the first-order optimality condition for \eqref{eq:defminF} is both necessary and sufficient; moreover, \eqref{eq:defminF} has a unique optimal solution. We next relate it to the equilibrium point of $\cD_1$.

The subdifferential  of $F(q)$ is given by:
\bqn
\partial F (q) & = & \partial C(q) + Xq + \Delta \tilde v\\
& = & \partial C(q) + (Xq + \tilde v) - v^{\text{nom}},
\eqn
where, by the definition of $C_i(q_i)$,
\bqn
\partial C(q) & = & \big[ \partial C_1(q_1)\ , \ \dots\ , \ \partial C_n(q_n) \big]^{\top}
\eqn
with 

\bqn
\partial C_i(q_i) & = & \left\{ \begin{array}{lcl}
		-f_i^{-1}(q_i)  & \text{if} & q_i 
		\neq 0 ,
		\\
		\left[- \frac{\delta_i}{2},~\frac{\delta_i}{2} \right] & \text{if} & q_i = 0.
		\end{array}  \right. 
\eqn

By the optimality condition, 
$q^*$ is an optimum of \eqref{eq:defminF} \textit{iff} there exists a (sub)gradient $\nabla F(q^*)\in \partial F (q^*)$ such that 
\bqn
\nabla F(q^*)^{\top}(q-q^*)\geq 0,\ \forall q\in\Omega ,
\eqn
which is equivalent to:
\bqn
q^* = [f \left( Xq^* + \tilde v - v^{\text{nom}} \right)]_{\Omega}.
\eqn

It follows that a point $(v^*, q^*)$ is an equilibrium of $\cD_1$ if and only if $q^*$ solves
\eqref{eq:defminF} and $v^* = Xq^* + \tilde{v}$. The existence and uniqueness of the optimal solution of \eqref{eq:defminF} then implies that of the equilibrium $(v^*, q^*)$.
\end{proof}

With $v=Xq+\tilde v$, the objective can be equivalently written as:
\begin{eqnarray}
\hspace{0mm}F(q, v)&\hspace{-2mm}=\hspace{-2mm}&C(q) + \frac{1}{2}(v-v^{\text{nom}})^{\top}X^{-1} (v-v^{\text{nom}}) -\frac{1}{2} \Delta\tilde v^{\top} X^{-1} \Delta\tilde v.\nonumber\\\label{eq:tradeoff}
\end{eqnarray}
Notice that the last term is a constant. Therefore, the local Volt/VAR control $\cD_1$ seeks an optimal trade-off between minimizing the cost of reactive power provisioning $C(q)$ and minimizing the cost of voltage deviation $\frac{1}{2}(v-v^{\text{nom}})^{\top}X^{-1} (v-v^{\text{nom}})$. 

\subsubsection{Further Characterization of Equilibrium}


The first term $C(q)$ of the objective \eqref{eq:tradeoff} is well-defined and has the desired additive structure. It is however not clear what specific structure the second term  $\frac{1}{2}(v-v^{\text{nom}})^{\top}X^{-1} (v-v^{\text{nom}})$ entails. We will further characterize this term in this subsection. 


Notice that bus $0$ has a fixed voltage magnitude, which decouples different subtrees rooted at it. Therefore, without loss of generality we only consider a topology where  the bus $0$ is of degree 1. 
Denote $\mathcal{T}$ the (sub)tree rooted at bus $1$ and $\hL_{T}$ the set of links of $\mathcal{T}$. Define an inverse tree $\mathcal{T}'$ that has the same sets of buses and lines as $\mathcal{T}$ but with reciprocal line reactance ${1}/{x_{ij}}$.
Let $\mathbb{L}\in\mathbb{R}^{n\times n}$ be the weighted Laplacian matrix of $\mathcal{T}'$ defined as follows:
\begin{eqnarray}
\mathbb{L}_{ij}=\left\{ \begin{array}{ll}
-{1}/{x_{ij}}, & (i,j) \in \hL_{T},\\ 
\sum_{(i,k) \in \hL} {1}/{x_{ik}}, &i = j,\\
0, &\text{otherwise.}
\end{array}\right.\nonumber
\end{eqnarray}
Recall that $x$ denotes the reactance of the line connecting buses $0$ and $1$, we have the following result by Liu {\em et al.} \cite{liu2017signal}.
\begin{theorem}[from \cite{liu2017signal}]\label{the:Xinverse}
Given the tree graph $\mathcal{G}=\{\hN \cup\{0\}, \hL\}$ with bus $0$ being of degree $1$ and its reactance matrix $X$ defined by \eqref{X_def},  the inverse matrix $X^{-1}$ has the following \textbf{explicit} form:
 \begin{eqnarray}\label{eq:x-1}
 X^{-1}=\mathbb{L}+\begin{bmatrix}
		1/x & 0 & \cdots & 0\\
		0& 0 & \cdots & 0\\
		\vdots&\vdots&\ddots&\vdots\\
		0 & 0 & \cdots&0
		\end{bmatrix}.
 \end{eqnarray}
\end{theorem}

With the above result, the cost function \eqref{eq:tradeoff} can be rewritten as: 
\begin{eqnarray}
F(q,v)&=&C(q)+\frac{1}{2}\left(\frac{(v_1-v^{\text{nom}})^2}{x}+\sum_{(i,j)\in\mathcal{L}_{T}}\frac{(v_i-v_j)^2}{x_{ij}}\right)\nonumber\\
&&\hspace{4mm}-\frac{1}{2} \Delta\tilde v^{\top} X^{-1} \Delta\tilde v.\label{eq:Fq2}
\end{eqnarray}
whose second term (i.e., the cost of voltage deviation) consists of two parts: the first part ${(v_1-v^{\text{nom}})^2}/{x}$ represents the cost of voltage deviation of the bus~1 from the nominal value, and the second part $\sum_{(i,j)\in\mathcal{L}_{T}}{(v_i-v_j)^2}/{x_{ij}}$ gives the cost of voltage deviation between the neighboring buses. This leads to a nice leader-follower structure where the first bus (the bus $1$) aims to attain the nominal voltage while each other bus tries to achieve the same voltage as that of the bus ``in front of'' it.

\subsection{Dynamics}\label{subsec:dynamics}
We now study the dynamic properties of the local Volt/VAR control $\cD_1$.

\begin{theorem}	\label{thm:cm}
	Suppose Assumptions~\ref{A1}--\ref{A2} hold.
	If 
	\begin{eqnarray}
		 \sigma_{\max}(\overline{A}X)<1,	\label{eq:cm1}
	\end{eqnarray}
	then the local Volt/VAR control $\cD_1$ converges to the unique equilibrium point $(v^*, q^*)$. Moreover, it converges exponentially fast to the equilibrium. 
\end{theorem}
\begin{proof}
	Write $\cD_1$ equivalently as a mapping $g_1$:
	\begin{eqnarray}
	q(t+1) = g_1(q(t)) :=  \big[f( Xq(t) + \Delta \tilde v -v^{\text{nom}})\big]_{\Omega}.\label{eq:dynamic1}
	\end{eqnarray}
	By Lemma~\ref{lemma:lipschitz} and the non-expansiveness property of projection operator, given any feasible $q,q'$ we have
	\begin{eqnarray}
		\left\| g_1(q) - g_1(q') \right\|_2  \leq M\|q - q'\|_2,
       \end{eqnarray}
	 where $M = \sigma_{\max}(\overline{A}X)$. When condition \eqref{eq:cm1} holds, $M<1$ and thus the mapping $g_1$ is a contraction, implying that $(v(t), q(t))$ converges exponentially to the unique equilibrium point under $\cD_1$.
\end{proof}

We next develop a sufficient condition for \eqref{eq:cm1}, which is easier to verify in practice.
Define the following matrix norms for some $W\in\mathbb{R}^{m\times n}$:
\begin{eqnarray}
&&\|W\|_1=\max_{1\leq j\leq n}\sum_{i=1}^m|w_{ij}|,\ \ 
\|W\|_{\infty}=\max_{1\leq i\leq m}\sum_{j=1}^n|w_{ij}|,\nonumber\\
&&\|W\|_{2}=\sqrt{\lambda_{max}(W^{\top}W)}=\sigma_{max}(W),\nonumber
\end{eqnarray}
where $\lambda_{max}(\cdot)$ denotes the largest eigenvalue of a matrix.
By utilizing the following relationship among these matrix norms based on H\"{o}lder's inequality
\begin{eqnarray}
\|W\|_2\leq\sqrt{\|W\|_1\cdot\|W\|_{\infty}}\ ,\label{eq:holder}
\end{eqnarray}
we have the following sufficient condition for convergence of $\cD_1$.

\begin{corollary}
Suppose Assumptions~\ref{A1}--\ref{A2} hold. If
	\begin{eqnarray}
	\max_{i\in\cN}(\overline{\alpha}_i)\cdot\max_{i\in\cN}\left( \sum_{j\in\cN} X_{ij} \right)<1, 	\label{eq:cm2}
	\end{eqnarray}
	then  $\cD_1$ converges exponentially fast to the unique equilibrium point $(v^*, q^*)$.
\end{corollary}
\begin{proof}
A sufficient condition for \eqref{eq:cm1} based on \eqref{eq:holder} is
\begin{eqnarray}
\|\overline{A}X\|_1<1\ \ \text{and}\ \ \|\overline{A}X\|_{\infty}<1.\label{eq:holder2}
\end{eqnarray}
Given symmetric matrix $X$, \eqref{eq:cm2} is thereafter sufficient for \eqref{eq:holder2}.
\end{proof}

\subsection{Limitation of the Non-Incremental Control}\label{sec:disslope}
The local voltage control \eqref{dyn:D1b} is non-incremental, as it decides the total amount of reactive power (instead of the change in reactive power) based on the deviation of current voltage from the nominal value. Intuitively, such a control may lead to over-actuation and oscillatory behavior. In order to have converging or stable behavior, the control function should not be too aggressive, i.e.,  have small (absolute) derivative. This can also be seen from Theorem~\ref{thm:cm}, and in the case of the piece-wise linear control function (\ref{eq:plf}),  implies a small  $\alpha_i$ value. 

On the other hand, seen from the equivalent objective \eqref{eq:tradeoff}, smaller cost functions $C_i(q_i)$ are preferred for better voltage regulation. However, a small cost function implies large derivative of the control function; see, e.g., the cost function \eqref{eq:integral}  that becomes smaller as $\alpha_i$ takes larger value, as well as the numerical examples in Section~\ref{sec:numerical} . 

Hence, there is a contention or fundamental limitation for the non-incremental control: control function with smaller derivative is preferred for convergence, while for better voltage regulation at the equilibrium control function with larger derivative is desired. This motivates us to seek new local voltage control schemes that are not subject to such a limitation, as will be seen in the next section.

\section{Forward Engineering: Decoupling Equilibrium and Dynamical Properties}\label{sec:forward1}
The optimization-based model \eqref{eq:defminF} does not only provide a way to characterize the equilibrium of the local voltage control (see Theorem~\ref{thm:eq}), but also 
suggests a principled way to engineer the control. New design goal such as fairness and economic efficiency can be taken incorporated by engineering the objective function in \eqref{eq:defminF}, and more importantly, new control schemes with better dynamical properties can be designed based on various optimization algorithms such as the (sub)gradient algorithm. In this section, we apply two iterative optimization algorithms to design local voltage control schemes that can decouple the dynamical property from the equilibrium property and have less restrictive convergence conditions than the non-incremental local voltage control studied in the previous section. 

\subsection{Local Voltage Control Based on the (Sub)gradient Algorithm}\label{sect:sg}
Given an optimization problem, we may apply different algorithms to solve it. A common algorithm that often admits distributed implementation is the (sub)gradient method \cite{BoVa04}.  Applying it to the problem \eqref{eq:defminF} leads to the following voltage control:
\begin{eqnarray}
q_i(t+1) &=& \left[ q_i(t) - \gamma_2 \frac{\partial F(q(t))}{\partial q_i}\right]_{\Omega_i},~i\in\hN,\label{eq:ica}
\end{eqnarray}
where $\gamma_2>0$ is the constant stepsize and the (sub)gradient is calculated as follows:

\begin{equation} 
\frac{\partial F(q(t))}{\partial q_i} 
= \left\{ \begin{array}{ll}
-f_i^{-1} (q_i(t)) + v_i(t)-v^{\text{nom}}&\mbox{if}~q_i(t) \neq 0, \\
v_i(t)-v^{\text{nom}}&\mbox{if}~q_i(t) = 0~\text{and}\\
&\hspace{-10mm}-{\delta}/{2} \leq v_i(t)-v^{\text{nom}} \leq {\delta_i}/{2}, \\
-{\delta_i}/{2} + v_i(t)-v^{\text{nom}} &\mbox{if}~q_i(t) = 0~\text{and}\\
&\hspace{-5mm} v_i(t)-v^{\text{nom}} > {\delta_i}/{2}, \\
{\delta_i}/{2} + v_i(t)-v^{\text{nom}} &\mbox{if}~q_i(t) = 0~\text{and}\\
&\hspace{-5mm}v_i(t)-v^{\text{nom}} <  -{\delta_i}/{2}.  \end{array} \right. \label{eq:sg}      
\end{equation}
The above control is \emph{incremental} as the change in reactive power (instead of the total reactive power) is based on the voltage deviation from the nominal value.  It is also distributed, since the decision at each bus $i\in\hN$ depends only on its current reactive power $q_i$ and voltage $v_i$. 

We thus obtain the following dynamical system $\cD_2$:
\vspace{-2mm}
\begin{subequations}
	\begin{empheq}[left=\empheqlbrace]{align}
	\hspace{2mm}v(t)\hspace{2mm}&=\hspace{2mm}Xq(t)+\tilde{v}\label{dyn:D3a}\\
	\hspace{2mm}q_i(t+1)\hspace{2mm}&=\hspace{2mm}\left[ q_i(t) - \gamma_2 \frac{\partial F(q(t))}{\partial q_i}\right]_{\Omega_i},~i\in\hN.\label{dyn:D2b}	\end{empheq} 
\end{subequations}


The following result is immediate.
\begin{theorem}
	\label{thm:eqn}
	Suppose Assumption~\ref{A1} holds.  Then there exists a unique equilibrium point for the dynamical system $\cD_2$. 
	Moreover, a point $(v^*, q^*)$ is an equilibrium if and only if $q^*$ is the unique
	optimal solution of problem \eqref{eq:defminF} and $v^* = Xq^* + \tilde{v}$.
\end{theorem} 





Since the feasible sets are bounded, we also have the bounded (sub)gradient of  $F(q)$ with some constant $G>0$:
\begin{eqnarray}
\|\nabla_q F(q)\|_2\leq G,\ \forall q\in\Omega.\label{eq:boundG}
\end{eqnarray}

\begin{theorem}\label{the:subgraconv}
Suppose Assumption~\ref{A1} hold. The dynamical system $\cD_2$ converges as
	\begin{eqnarray}
	\limsup_{t\rightarrow\infty} \sum_{\tau=1}^t\frac{F(q(\tau))-F(q^*)}{t}=\gamma_2^2G^2.\label{eq:subgconv}
	\end{eqnarray}
\end{theorem}
\begin{proof}
We characterize the distance between $q(t+1)$ and $q^*$ as:
	\begin{eqnarray}
		&&\|q(t+1)-q^*\|_2^2\nonumber\\
		&\leq& \|q(t)-\gamma_2\nabla_q F(q(t))-q^*\|_2^2\nonumber\\
		&=& \|q(t)-q^*\|_2^2+\gamma^2_2\|\nabla_q F(q(t))\|_2^2-2\gamma_2(q(t)-q^*)^{\top}\nabla_q F(q(t))\nonumber\\[2pt]
		&\leq&\|q(t)-q^*\|_2^2+ \gamma_2^2G^2-(F(q(t))-F(q^*))\nonumber\\[-3pt]
		&\leq&\|q(1)-q^*\|_2^2+ t\gamma_2^2G^2-\sum_{\tau=1}^t(F(q(\tau))-F(q^*)),\nonumber
	\end{eqnarray}
	where the first inequality is due to non-expansiveness of projection operator, the second inequality is because of the definition of subgradient as well as the bounded gradient \eqref{eq:boundG}, and the last inequality is by repeating previous steps.
	
	As $\|q(t+1)-q^*\|_2^2\geq 0$, it follows that:
	\begin{eqnarray}
	\sum_{\tau=1}^t\frac{F(q(\tau))-F(q^*)}{t}\leq \|q(1)-q^*\|_2^2/t+ \gamma_2^2G^2.
	\end{eqnarray}
	When $t\rightarrow\infty$, we have \eqref{eq:subgconv}.
\end{proof}


%

Notice that for any control functions $f_i$ (that satisfies Assumptions~\ref{A1}--\ref{A2}), one can always find a small enough stepsize $\gamma_2$ such that $\cD_2$ converges to a neighborhood of the $(v^*, q^*)$ of required accuracy on running average. Moreover, as shown in \cite{farivar2015local}, when $q^*$ is not close to the non-differentiable point, $\cD_2$ converges exactly to the optimum. In contrast, the convergence condition \eqref{eq:cm1} for the non-incremental voltage control $\cD_1$ does constrain the allowable control functions $f_i$. Therefore, $\cD_2$ permits better voltage regulation than $\cD_1$; see the discussion at the end of Section~\ref{sec:eq} and the simulation results in Fig.~\ref{fig:slopes}

Nevertheless, the (sub)gradient nature of $\cD_2$ may prevent it from converging to the exact optimal point. This could happen when the optimum is close to the non-differentiable point ($q^*=0$ in this case) where the value of subgradient \eqref{eq:sg} changes abruptly if $\delta_i\neq 0$. Moreover,  the (sub)gradient computation \eqref{eq:sg} requires computing the inverse of the control function $f_i$, which can be computationally expensive for general control functions, as well as tracking the value of $v_i$ with respect to deadband $\pm\delta_i/2$. These limitations motivate us to design another incremental control scheme with better convergence and lower implementation complexity.

\subsection{Local Voltage Control Based on the Pseudo-Gradient Algorithm}\label{sec:forward2}
The pseudo-gradient can provide a good search direction for an optimization problem without requiring the objective function to be differentiable; see, e.g., \cite{wen2004optimal}. Applying it to the problem \eqref{eq:defminF} leads to the following incremental voltage control at bus $i\in\cN$:   
\begin{subequations}\label{eq:pg}
\begin{eqnarray}
\hspace{-6mm}q_i(t+1) \hspace{-2.5mm}&=&\hspace{-2.5mm}\big[ q_i(t) - \gamma_3 \big(q_i(t)- f_i(v_i(t)-v_i^{\text{nom}})\big)\big]_{\Omega_i},  \label{eq:pga}\\
\hspace{-8.5mm}&=&\hspace{-2.5mm}  \big[(1-\gamma_3)q_i(t)+\gamma_3 f_i(v_i(t)-v_i^{\text{nom}})\big]_{\Omega_i}\label{eq:pgb}
\end{eqnarray}
\end{subequations}
where $\gamma_3 >0$ is a constant stepsize/weight and $q_i- f_i(v_i-v_i^{\text{nom}})$ is the pseudo-gradient. The above control is distributed, and is simpler to implement than the control \eqref{eq:ica}.

With \eqref{eq:pg} we obtain the following dynamical system $\cD_3$:
\begin{subequations}
	\begin{empheq}[left=\empheqlbrace]{align}
	\hspace{0mm}v(t)\hspace{0mm}&=\hspace{0mm}Xq(t)+\tilde{v},\label{dyn:D3a}\\
	\hspace{0mm}q_i(t+1)\hspace{0mm}&=\hspace{0mm}\left[ q_i(t) - \gamma_3 \Big(q_i(t)- f_i\big(v_i(t)-v_i^{\text{nom}}\big)\Big)\right]_{\Omega_i},\nonumber\\
	&\hspace{50mm}i\in\hN.\label{dyn:D3b}
	\end{empheq} 
\end{subequations}
Notice that $\cD_3$ has the same equilibrium condition as the dynamical systems $\cD_1$ and $\cD_2$.  The following result is immediate.
\begin{theorem}
	\label{thm:eqnpg}
	Suppose Assumption~\ref{A1} holds.  There exists a unique equilibrium point for the dynamical system $\cD_3$. 
	Moreover, a point $(v^*, q^*)$ is an equilibrium if and only if $q^*$ is the unique
	optimal solution of problem \eqref{eq:defminF} and $v^* = Xq^* + \tilde{v}$.
\end{theorem} 

We now analyze the convergence of the dynamical system $\cD_3$. We first introduce the following useful results. 

Denote by $\nabla_v f$ the diagonal matrix with each entry $\big(\nabla_v f\big)_{ii}$ representing the (sub)gradient defined as
	\begin{eqnarray}
	(\nabla_v f)_{ii}~~\left\{\begin{array}{ll}
	= f_i'(v_i) & \mbox{if}~ v_i\in(-\infty, -\delta_i/2)\cup\\
	& (-\delta_i/2,\delta_i/2) \cup (\delta_i/2, +\infty)\\
	\in [f_i'(v_i^-), f_i'(v_i^+)] & \mbox{if}~ v_i=-\delta_i/2\\
	\in [f_i'(v_i^+), f_i'(v_i^-)] & \mbox{if}~ v_i=\delta_i/2
	\end{array}\right.,
	\end{eqnarray}
which is bounded as $-\overline{\alpha}_i\leq (\nabla_v f)_{ii}\leq 0$ based on Assumptions~\ref{A1}--\ref{A2}. Denote by $\lambda$ any eigenvalue of the matrix $\nabla_vf X$. Consider $\nabla_vf X$'s similar matrix $X^{1/2}\nabla_vfX^{1/2}$, which is symmetric and negative semidefinite with real and nonpositive eigenvalues. Therefore, eigenvalues of the original \emph{asymmetric} matrix $\nabla_vfX$ are also \emph{real} and \emph{nonpositive}. Similarly, all the eigenvalues of $\overline{A}X$ are \emph{real} and \emph{positive}.

\begin{theorem}\label{thm:mcon2}
	Suppose Assumptions~\ref{A1}--\ref{A2} hold. If the stepsize $\gamma_3$ satisfies the following condition: 
	\begin{eqnarray}
		0<\gamma_3<2/\big(1+\lambda_{max}(\overline{A}X)\big),\label{eq:d3_0}
	\end{eqnarray}
	then the dynamical system $\cD_3$ converges to its unique equilibrium point.
\end{theorem}
\begin{proof}
	Write $\cD_3$ equivalently as a mapping $g_3$:
	\begin{eqnarray}
	q(t+1) = g_3(q(t)) := \big[ (1-\gamma_3)q(t)+\gamma_3 f(v(q(t)))\big]_{\Omega}.	\label{eq:dynamic3}
	\end{eqnarray}
The Jacobian matrix of $g_3$ without projection operator is computed as
\begin{eqnarray}
	\nabla_q g_3=(1-\gamma_3) \mathcal{I}_N+\gamma_3 \nabla_v f X,
\end{eqnarray}
where $\mathcal{I}_N$ is $N\times N$ identity matrix. On the other hand, when projection operator is active for some $q_i$, the $i$th row of the resulting $\nabla_q g_3$ is all 0. So, by Gershgorin circle theorem \cite{varga2009matrix}, the magnitude of $\nabla_q g_3$'s eigenvalue without active projection has a larger bound than that with projection. Thus it is sufficient to consider $g_3$ without the projection operator in this proof.

	Denote by $z$ the eigenvector of matrix $\nabla_v f X$ corresponding to $\lambda$. By definition one has $\nabla_v f  X z = \lambda z$. Therefore,
	\begin{eqnarray}
		\nabla_q g_3 z=(1-\gamma_3+\gamma_3\lambda)z,
	\end{eqnarray}
	that is, the corresponding eigenvalue of $\nabla_q g_3$ with respect to $\lambda$ is $1-\gamma_3+\gamma_3\lambda$. 
	To ensure that $g_3$ is stable, one must have \cite{jahangiri2013distributed, galor2007discrete}
	\begin{eqnarray}
		-1<1-\gamma_3+\gamma_3\lambda<1\label{eq:eigenvalue}
	\end{eqnarray} 
	for any eigenvalue $\lambda$ of $\nabla_v f X$. \eqref{eq:d3_0} is sufficient for the left-hand side of \eqref{eq:eigenvalue} and the right-hand side always holds because $\lambda$ is nonpositive. This completes the proof.
\end{proof}


We conclude that the voltage control $\cD_3$ based on the pseudo-gradient algorithm converges to the optimum with properly chosen stepsize, while the control $\cD_2$ based on the subgradient algorithm usually converges to within a small neighborhood of the optimum on running average.

\begin{remark}
Notice that, when  $\gamma_3\leq 1$ in, the control  \eqref{eq:pg} has a nice interpretation of the new decision $q_i(t+1)$ being a (positively-)weighted sum of the decision $q_i(t)$ at the previous time and the local control $u_i(t)=f_i(v_i(t)-v_i^{\text{nom}})$ in reactive power. Similar approaches in literature are also identified as exponentially weighted moving average (EWMA) control and delayed control, etc. However, here we do not require $\gamma_3\leq 1$ for $\cD_3$ to converge, as long as the condition~\eqref{eq:d3_0} is satisfied.
\end{remark}



\begin{figure}
	\centering
	\includegraphics[scale=0.265]{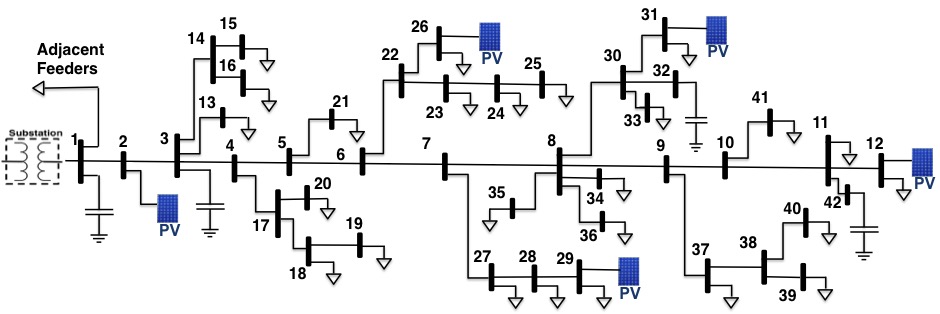}
	\caption{Circuit diagram for SCE distribution system.}
	\label{42bus}
\end{figure}

\begin{table*}
	\centering
	\begin{tabular}{|c|c|c|c|c|c|c|c|c|c|c|c|c|c|c|c|c|c|}
		\hline
		\multicolumn{18}{|c|}{Network Data}\\
		\hline
		\multicolumn{4}{|c|}{Line Data}& \multicolumn{4}{|c|}{Line Data}& \multicolumn{4}{|c|}{Line Data}& \multicolumn{2}{|c|}{Load Data}& \multicolumn{2}{|c|}{Load Data}&\multicolumn{2}{|c|}{\!\!PV Generators\!\!}\\
		\hline
		From&To&R&X&From&To& R& X& From& To& R& X& Bus& Peak & Bus& Peak & Bus&\!\!Capacity\!\!\\
		Bus.&Bus.&$(\Omega)$& $(\Omega)$ & Bus. & Bus. & $(\Omega)$ & $(\Omega)$ & Bus.& Bus.& $(\Omega)$ & $(\Omega)$ & No.&  MVA& No.& MVA& No.& MW\\
		\hline
		1	&	2	&	0.259	&	0.808	&	8	&	34	&	0.244	&	0.046 	&	18	&	19	&	0.198	&	 0.046	&	11	&	 0.67	&	28	&	 0.27 	&				 &		 \\

		2	&	3	&	0.031	&	0.092	&	8	&	36	&	0.107	 &	0.031 	&	22	&	26	&	0.046	&	 0.015	&	12	&	0.45		 &	29	&	0.2 &			 2	&	 1\\

		3	&	4	&	0.046	&	0.092	&	8	&	30	&	0.076	 &	0.015 	&	22	&	23	&	0.107	&	 0.031	&	13	&	0.89	 &	31	&	0.27	  &			26	&	 2	 \\

		3	&	13	&	0.092	&	0.031	&	8	&	9	&	0.031	 &	0.031 	&	23	&	24	&	0.107	&	 0.031	&	15	&	0.07	 &	33	&	0.45	 	 &			29	&	 1.8 	 \\

		3	&	14	&	0.214	&	0.046	&	9	&	10	&	0.015	 &	0.015	&	24	&	25	&	0.061	&	 0.015	&	16	&	0.67	 &	34	&	1.34  &			31	&	 2.5 	 \\

		4	&	17	&	0.336	&	0.061	&	 9	&	37	&	0.153	 &	0.046 	&	27	&	28	&	0.046	&	 0.015	&	18	&	0.45	 &	35	&	0.13	  &			12	&	 3 	 \\

		4	&	5	&	0.107	&	0.183	&	10	&	11	&	0.107	 &	0.076 	&	28	&	29	&	0.031	&	0		&	19	&	1.23	 &	36	&	0.67	  &			&		 \\

		5	&	21	&	0.061	&	0.015	&	10	&	41	&	0.229	 &	0.122 	&	30	&	31	&	0.076	&	 0.015  &	20	&	0.45	 &	37	&	0.13	 	&			 &			 \\

		5	&	6	&	0.015	&	0.031	&	11	&	42	&	0.031	 &	0.015 	&	30	&	32	&	0.076	&	 0.046	&	21	&	0.2 &	39	&	0.45 	&			  &	 \\

		6	&	22	&	0.168	&	0.061	&	11	&	12	&	0.076	 &	0.046 	&	30	&	33	&	0.107	&	 0.015	&	23	&	0.13		 &	40	&	0.2 	&		 &	\\

		6	&	7	&	0.031	&	0.046	&	14	&	16	&	0.046	 &	0.015 	&	37	&	38	&	0.061	&	0.015	&	24	&	0.13	 &	 41	&	0.45		&		&		 \\		\cline{15-18}

		7	&	27	&	0.076	&	0.015	&	14	&	15	&	0.107	 &	0.015	&	38	&	39	&	0.061	&	0.015	&	 25	&	0.2 &	 \multicolumn{4}{c|}{$V_{base}$ = 12.35 KV}		 	\\

		7	&	8	&	0.015	&	0.015	&	17	&	18	&	0.122	 &	0.092	&	38	&	40	&	0.061	&	 0.015	&	 26	&	0.07 &	 \multicolumn{4}{c|}{$S_{base}$ = 1000 KVA} 	 	\\

		8	&	35	&	0.046	&	0.015	&	17	&	20	&	0.214	 &	0.046	&	 	&	 	&	 	&	 	&	27	&	0.13		 &	 \multicolumn{4}{c|}{$Z_{base}$ = 152.52 $\Omega$}   	 	 \\	
		
		\hline
	\end{tabular}
	\caption{Network Parameters of the SCE Circuit: Line impedances, peak spot load KVA, Capacitors and PV generation's nameplate ratings. }
	\label{data}
\end{table*}

\section{Numerical examples}\label{sec:numerical}
Consider a distribution feeder of South California Edison (SCE) with a high penetration of photovoltaic (PV) generation.
As shown in Fig.~\ref{42bus}, bus 1 is the substation (root bus) and five PV generators are integrated at buses 2, 12, 26, 29, and 31. As we aim to study the Volt/VAR control through PV inverters, all shunt capacitors are assumed to be off. Table~\ref{data} contains the network data including the line impedance, the peak MVA demand of loads, and the capacity of the PV generators. It is important to note that all studies are run with a full AC power flow model with MATPOWER \cite{zimmerman2011matpower} instead of its linear approximation. As will be seen, the results we develop for the linearized model are corroborated numerically with the full power flow model. 

In all numerical studies, we implement homogeneous piecewise linear droop control functions~(\ref{eq:plf}) of the IEEE 1547.8 Standard \cite{standard1547a} for all PV inverters, with their deadbands from 0.98 p.u. to 1.02 p.u. and slopes $\alpha_i$ to be determined.

\begin{figure}
	\centering
	\includegraphics[trim = 0mm 0mm 0mm 0mm, clip, scale=0.33]{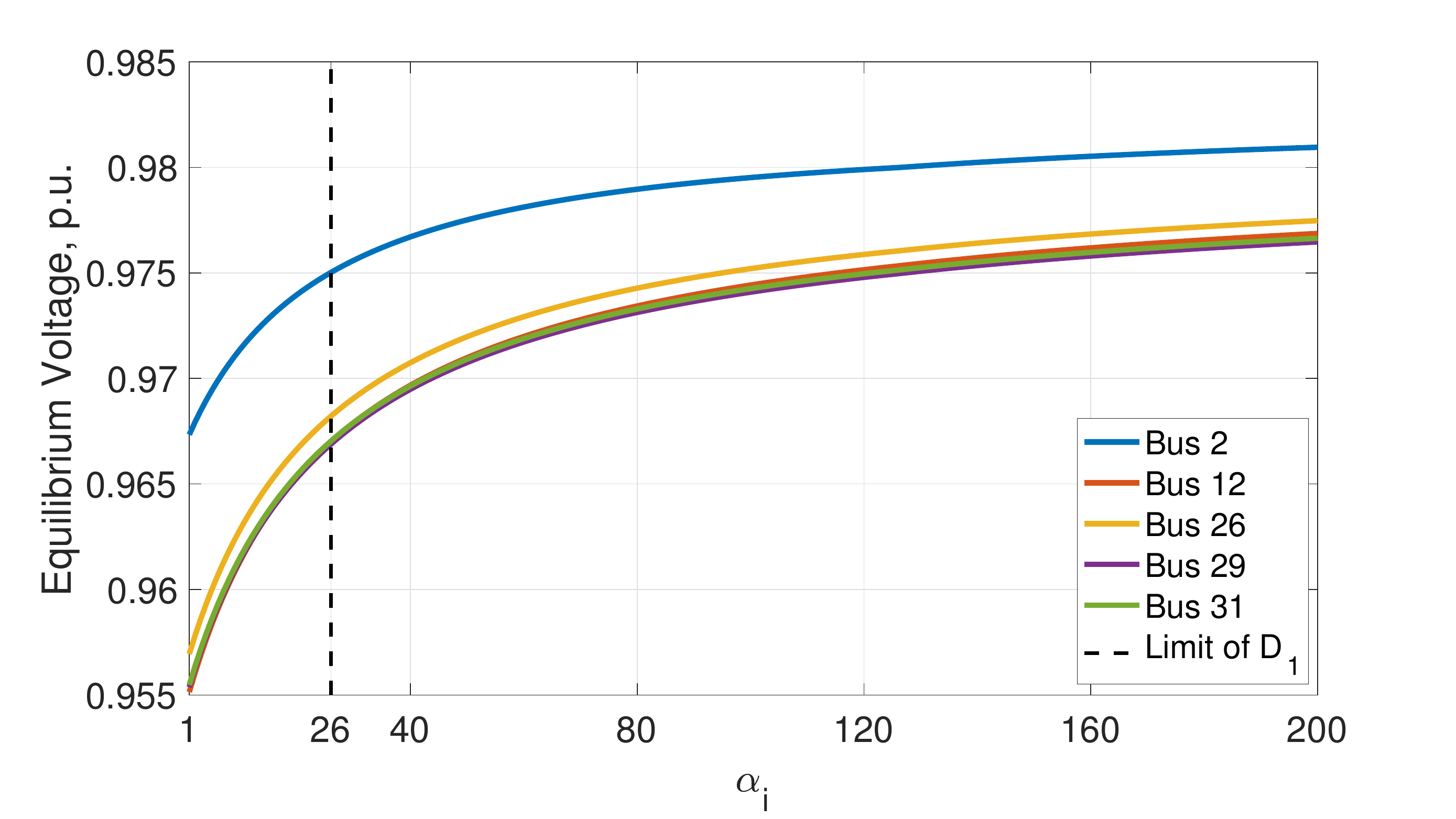} 
	\caption{Equilibrium voltage versus the $\alpha_i$ value: As $\alpha_i$ increases,  the equilibrium voltage $v_i^*$ deviates less from the nominal value.}	\label{fig:slopes}
\end{figure}

\subsection{Equilibrium}
As discussed in Section~\ref{sec:disslope}, large (absolute) slopes of the control function lead to better voltage regulation at the equilibrium. To show this, we change $\alpha_i$ from 1 to 200 and record the corresponding equilibrium voltages $v^*$. As shown in Fig.~\ref{fig:slopes}, $v^*$ gets closer to $v^{\text{nom}}$ as $\alpha_i$ increases. This confirms our previous discussion that steeper control functions are to be implemented in order to achieve smaller voltage deviation from the nominal value.



\subsection{Dynamics}
\subsubsection{Convergence of Dynamical System $\cD_1$}
As shown in Fig.~\ref{fig:d1converge}, the dynamical system $\cD_1$ displays less stable behavior as the control function become steeper with the increase of $\alpha_i$ value, till it ends up with oscillation when $\alpha_i$ becomes too large. See also the vertical dash line on Fig.~\ref{fig:slopes} beyond which there is no convergence. As discussed in Section ~\ref{sec:disslope}, there is a contention between convergence and equilibrium performance for the non-incremental voltage control \eqref{dyn:D1b}: A smaller (absolute) slope is preferred for convergence, while a larger one is preferred for voltage regulation at the equilibrium. 


\begin{figure}
	\centering
	\includegraphics[trim = 0mm 0mm 0mm 0mm, clip, scale=0.45]{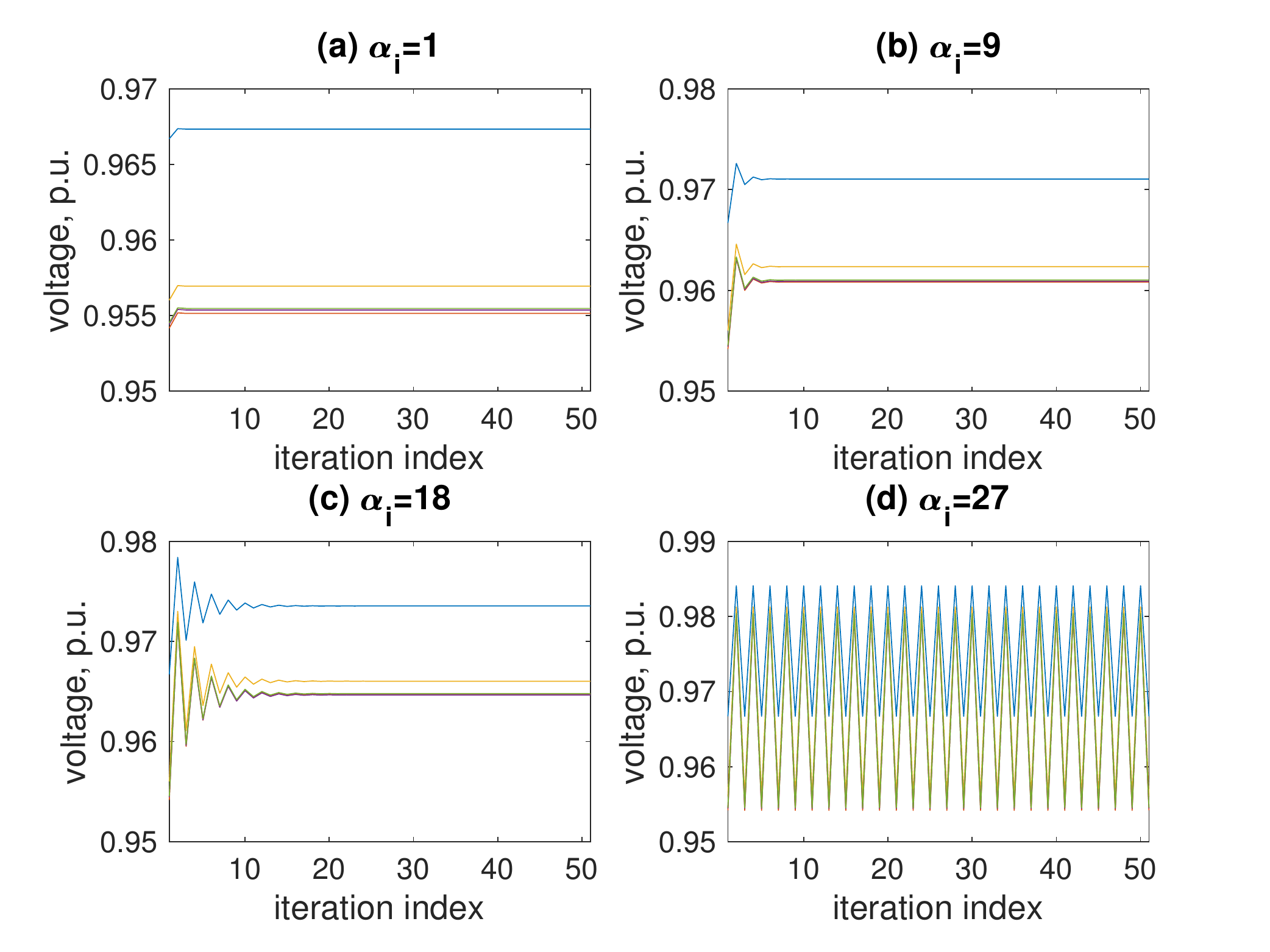} 
	\caption{Evolution of voltage of the dynamical system $\cD_1$ with different slopes of the piecewise linear control function: Voltage does not converge when the (absolute) slope of the control function become too large (when $\alpha_i>26$ in this example).}	\label{fig:d1converge}
\end{figure}

\subsubsection{Convergence of Dynamical Systems $\cD_2$ and $\cD_3$} As discussed in Section~\ref{sec:forward1}, given any control function, $\cD_2$ and $\cD_3$  converge if small enough stepsizes are chosen, and we can thus decouple the equilibrium property from the dynamical property. For instance, when $\alpha_i= 27$, the dynamical system $\cD_1$ does not converge; see Fig.~\ref{fig:d1converge}(b).  However, as shown in Fig.~\ref{fig:d2d3converge}, when the stepsizes $\gamma_2$ and $\gamma_3$  are properly chosen,  the dynamical systems $\cD_2$ and $\cD_3$ converge. 



\begin{figure}
	\centering
	\includegraphics[trim = 0mm 0mm 0mm 0mm, clip, scale=0.45]{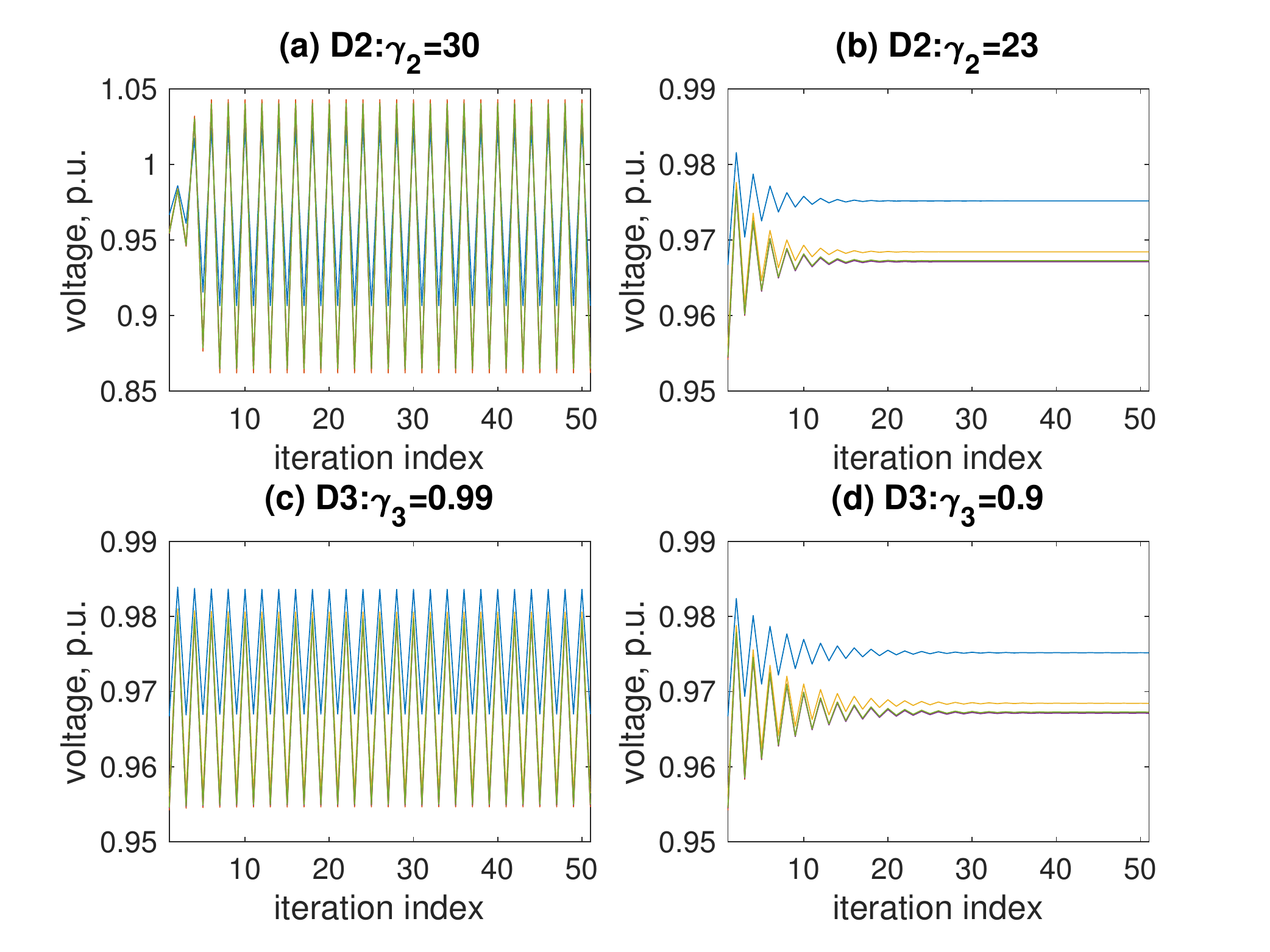} 
	\caption{Evolution of voltage of the dynamical systems $\cD_2$ and $\cD_3$ with $\alpha_i=27$:  Convergence is ensured with small enough stepsizes.}	\label{fig:d2d3converge}
\end{figure}

\begin{figure}
	\centering
	\includegraphics[trim = 0mm 0mm 0mm 0mm, clip, scale=0.3]{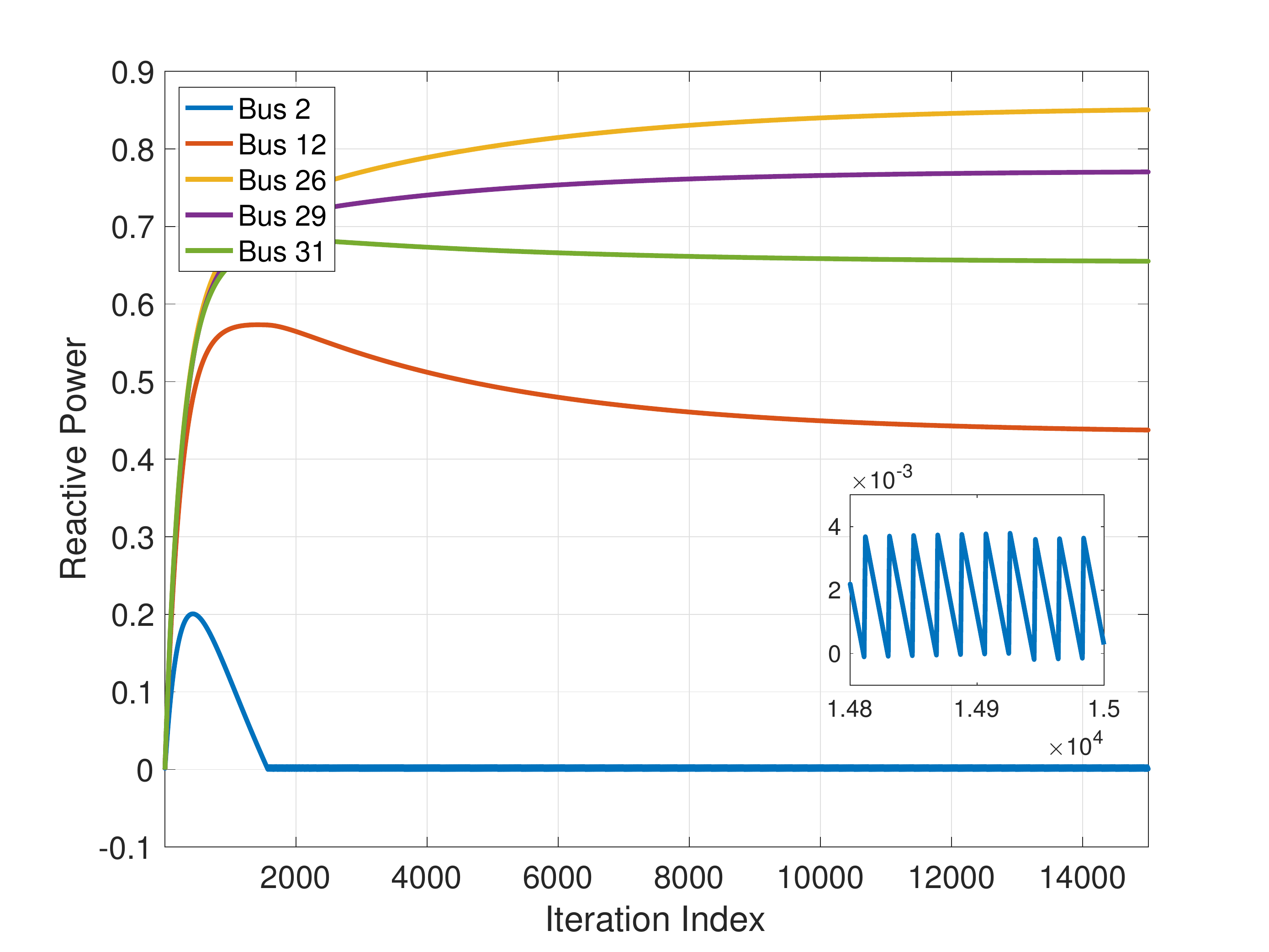} 
	\caption{Convergence of the dynamical system $\cD_2$ to within a small neighborhood of the equilibrium.}	\label{fig:d2converge}
\end{figure}

\begin{figure}
	\centering
	\includegraphics[trim = 0mm 0mm 0mm 0mm, clip, scale=0.3]{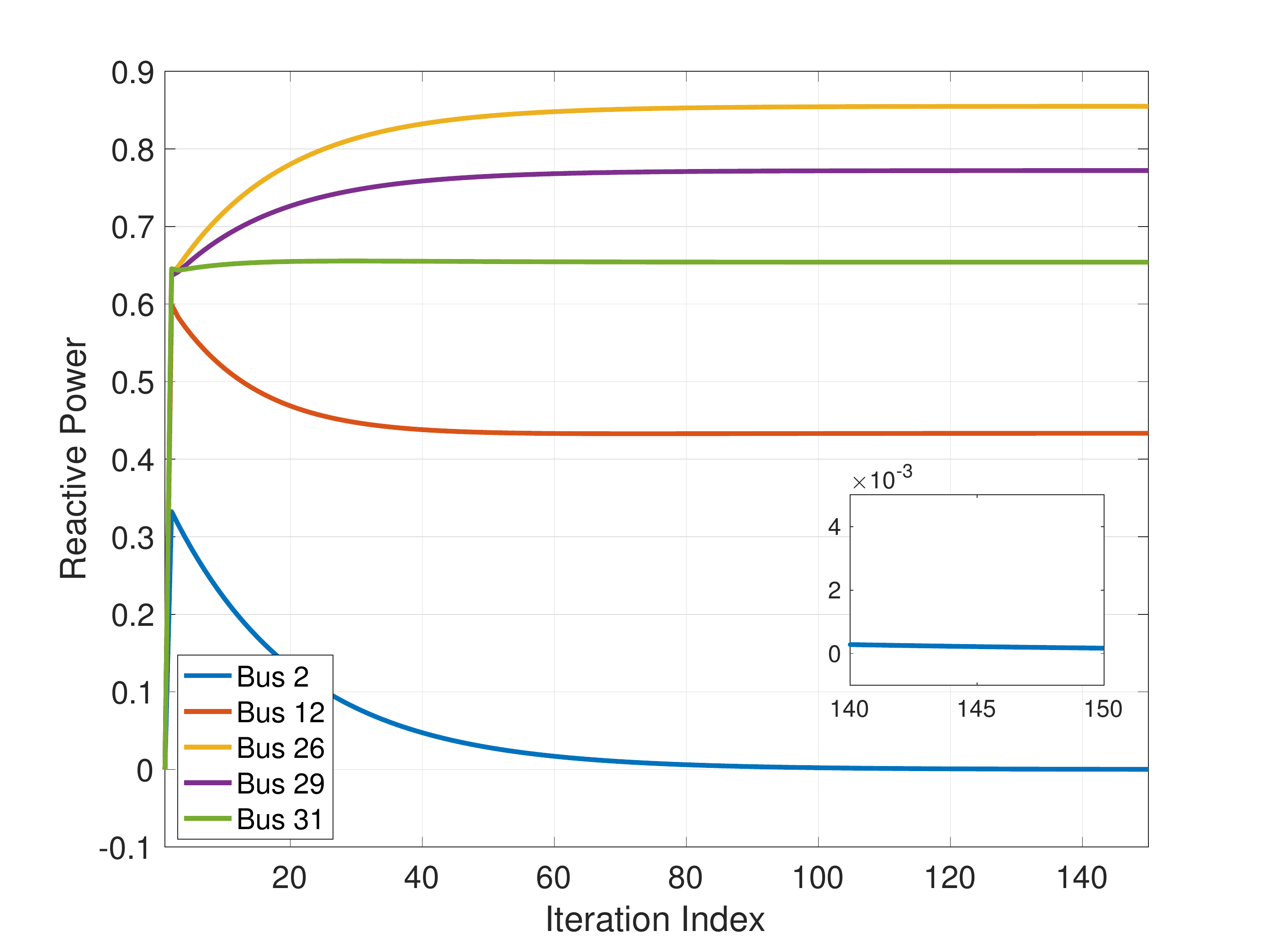} 
	\caption{Convergence of the dynamical system $\cD_3$ to the equilibrium.}	\label{fig:d3converge}
\end{figure}

\subsubsection{Convergence at Non-Differentiable Point} As discussed in Section~\ref{sect:sg}, the dynamical system $\cD_2$ based on subgradient algorithm can only converge to within a small neighborhood of the equilibrium  if it is at a nonsmoooth point of the objective function \eqref{eq:obj}. We tune the parameters such that the equilibrium reactive power provisioned at certain bus --- the bus 2 in this case --- is close to zero. As shown in Fig.~\ref{fig:d2converge},  $\cD_2$ eventually converges to a small region around the optimum, even with very small stepsize. On the other hand, as shown in Fig.~\ref{fig:d3converge}, the dynamical system $\cD_3$ based on pseudo-gradient algorithm converges to the equilibrium despite the non-smoothness of the objective function at the equilibrium. See the   embedded windows in Fig.~\ref{fig:d2converge}-\ref{fig:d3converge}. 

\section{Conclusion}\label{sec:conclusion}
We have investigated local voltage control with a general class of control functions in distribution networks. We show that the power system dynamics with non-incremental local voltage control can be seen as distributed algorithm for solving a well-defined optimization problem (reverse engineering). The reverse engineering further reveals a limitation of the non-incremental voltage control: the convergence condition is restrictive and prevents better voltage regulation at equilibrium. This motivates us to design two incremental local voltage control schemes based on different algorithms for solving the same optimization problem (forward engineering). The new control schemes decouple the dynamical property from the  equilibrium property, and have much less restrictive convergence conditions. 
This work presents another step towards developing a new foundation -- network dynamics as optimization algorithms -- for distributed realtime control and optimization of future power networks



\bibliographystyle{IEEEtran}
\bibliography{reverseforward.bib,ChenReferencesPower.bib}

\bibliographystyle{IEEEtran}
\bibliography{biblio.bib}

\begin{thebibliography}{10}
\providecommand{\url}[1]{#1}
\csname url@samestyle\endcsname
\providecommand{\newblock}{\relax}
\providecommand{\bibinfo}[2]{#2}
\providecommand{\BIBentrySTDinterwordspacing}{\spaceskip=0pt\relax}
\providecommand{\BIBentryALTinterwordstretchfactor}{4}
\providecommand{\BIBentryALTinterwordspacing}{\spaceskip=\fontdimen2\font plus
\BIBentryALTinterwordstretchfactor\fontdimen3\font minus
  \fontdimen4\font\relax}
\providecommand{\BIBforeignlanguage}[2]{{%
\expandafter\ifx\csname l@#1\endcsname\relax
\typeout{** WARNING: IEEEtran.bst: No hyphenation pattern has been}%
\typeout{** loaded for the language `#1'. Using the pattern for}%
\typeout{** the default language instead.}%
\else
\language=\csname l@#1\endcsname
\fi
#2}}
\providecommand{\BIBdecl}{\relax}
\BIBdecl



\bibitem{baran1989optimala}
M.~E. Baran and F.~F. Wu, ``Optimal capacitor placement on radial distribution
  systems,'' \emph{IEEE Trans. on Power Delivery}, vol.~4, no.~1, pp. 725--734,
  1989.

\bibitem{baran1989optimalb}
------, ``Optimal sizing of capacitors placed on a radial distribution
  system,'' \emph{IEEE Trans. on Power Delivery}, vol.~4, no.~1, pp. 735--743,
  1989.

\bibitem{mccrone2017global}
A.~McCrone, U.~Moslener, F.~d'Estais, and C.~Gr{\"u}ning, ``Global trends in
  renewable energy investment,'' 2017.

\bibitem{ren21}
REN21, ``Renewables 2017 global status report,'' 2017.

\bibitem{standard1547a}
{Standards Coordinating Committee 21 of Institute of Electrical and Electronics
  Engineers, Inc., IEEE Standard, P1547.8\texttrademark /D8}, ``Recommended
  practice for establishing methods and procedures that provide supplemental
  support for implementation strategies for expanded use of {IEEE Standard
  1547},'' \emph{IEEE ballot document}, Aug. 2014.

\bibitem{kekatos2015stochastic}
V.~Kekatos, G.~Wang, A.~J. Conejo, and G.~B. Giannakis, ``Stochastic reactive
  power management in microgrids with renewables,'' \emph{IEEE Trans. on Power
  Systems}, vol.~30, no.~6, pp. 3386--3395, 2015.

\bibitem{farivar2011inverter}
M.~Farivar, C.~R. Clarke, S.~H. Low, and K.~M. Chandy, ``Inverter var control
  for distribution systems with renewables,'' \emph{Proc. of IEEE International
  Conference on Smart Grid Communications (SmartGridComm)}, pp. 457--462, 2011.

\bibitem{lavaei2012zero}
J.~Lavaei and S.~H. Low, ``Zero duality gap in optimal power flow problem,''
  \emph{IEEE Trans. on Power Systems}, vol.~27, no.~1, pp. 92--107, 2012.

\bibitem{low2014convexa}
S.~H. Low, ``Convex relaxation of optimal power flo---{Part I: F}ormulations
  and equivalence,'' \emph{IEEE Trans. on Control of Network Systems}, vol.~1,
  no.~1, pp. 15--27, 2014.

\bibitem{low2014convexb}
------, ``Convex relaxation of optimal power flow---{Part II: E}xactness,''
  \emph{IEEE Trans. on Control of Network Systems}, vol.~1, no.~2, pp.
  177--189, 2014.

\bibitem{kekatos2016voltage}
V.~Kekatos, L.~Zhang, G.~B. Giannakis, and R.~Baldick, ``Voltage regulation
  algorithms for multiphase power distribution grids,'' \emph{IEEE Trans. on
  Power Systems}, vol.~31, no.~5, pp. 3913--3923, 2016.

\bibitem{baker2017network}
K.~Baker, A.~Bernstein, E.~Dall'Anese, and C.~Zhao, ``Network-cognizant voltage
  droop control for distribution grids,'' \emph{arXiv preprint
  arXiv:1702.02969}, 2017.

\bibitem{dall2014optimal}
E.~Dall'Anese, S.~V. Dhople, and G.~B. Giannakis, ``Optimal dispatch of
  photovoltaic inverters in residential distribution systems,'' \emph{IEEE
  Trans. on Sustainable Energy}, vol.~5, no.~2, pp. 487--497, 2014.

\bibitem{zhou2017incentive}
X.~Zhou, E.~Dall'Anese, L.~Chen, and A.~Simonetto, ``An incentive-based online
  optimization framework for distribution grids,'' \emph{IEEE Trans. on
  Automatic Control}, 2017.

\bibitem{zhou2017discrete}
X.~Zhou, E.~Dall'Anese, and L.~Chen, ``Online stochastic control of discrete
  loads in distribution grids,'' \emph{arXiv preprint arXiv:1711.09953}, 2017.

\bibitem{li2014real}
N.~Li, G.~Qu, and M.~Dahleh, ``Real-time decentralized voltage control in
  distribution networks,'' \emph{Proc. of Annual Allerton Conference on
  Communication, Control, and Computing (Allerton)}, pp. 582--588, 2014.

\bibitem{bolognani2013distributed}
S.~Bolognani and S.~Zampieri, ``A distributed control strategy for reactive
  power compensation in smart microgrids,'' \emph{IEEE Trans. on Automatic
  Control}, vol.~58, no.~11, pp. 2818--2833, 2013.

\bibitem{magnusson2017voltage}
S.~Magnusson, C.~Fischione, and N.~Li, ``Voltage control using limited
  communication,'' \emph{arXiv preprint arXiv:1704.00749}, 2017.

\bibitem{peng2016distributed}
Q.~Peng and S.~H. Low, ``Distributed optimal power flow algorithm for radial
  networks, {I}: {B}alanced single phase case,'' \emph{IEEE Trans. on Smart
  Grid}, 2016.

\bibitem{vsulc2014optimal}
P.~{\v{S}}ulc, S.~Backhaus, and M.~Chertkov, ``Optimal distributed control of
  reactive power via the alternating direction method of multipliers,''
  \emph{IEEE Trans. on Energy Conversion}, vol.~29, no.~4, pp. 968--977, 2014.

\bibitem{shi2015distributed}
W.~Shi, X.~Xie, C.~C. Chu, and R.~Gadh, ``Distributed optimal energy management
  in microgrids,'' \emph{IEEE Trans. on Smart Grid}, vol.~6, no.~3, pp.
  1137--1146, 2015.

\bibitem{magnusson2015distributed}
S.~Magn{\'u}sson, P.~C. Weeraddana, and C.~Fischione, ``A distributed approach
  for the optimal power-flow problem based on {ADMM} and sequential convex
  approximations,'' \emph{IEEE Trans. on Control of Network Systems}, vol.~2,
  no.~3, pp. 238--253, 2015.

\bibitem{bazrafshan2017decentralized}
M.~Bazrafshan and N.~Gatsis, ``Decentralized stochastic optimal power flow in
  radial networks with distributed generation,'' \emph{IEEE Trans. on Smart
  Grid}, vol.~8, no.~2, pp. 787--801, 2017.

\bibitem{wu2017distributed}
C.~Wu, G.~Hug, and S.~Kar, ``Distributed voltage regulation in distribution
  power grids: Utilizing the photovoltaics inverters,'' \emph{Proc. of American
  Control Conference (ACC), 2017}, pp. 2725--2731, 2017.

\bibitem{simpson2017voltage}
J.~W. Simpson-Porco, F.~D{\"o}rfler, and F.~Bullo, ``Voltage stabilization in
  microgrids via quadratic droop control,'' \emph{IEEE Trans. on Automatic
  Control}, vol.~62, no.~3, pp. 1239--1253, 2017.

\bibitem{zhu2016fast}
H.~Zhu and H.~J. Liu, ``Fast local voltage control under limited reactive
  power: {O}ptimality and stability analysis,'' \emph{IEEE Trans. on Power
  Systems}, vol.~31, no.~5, pp. 3794--3803, 2016.

\bibitem{zhou2016local}
X.~Zhou, J.~Tian, L.~Chen, and E.~Dall'Anese, ``Local voltage control in
  distribution networks: {A} game-theoretic perspective,'' \emph{Proc. of North
  American Power Symposium (NAPS)}, pp. 1--6, 2016.

\bibitem{zhou2016incremental}
X.~Zhou and L.~Chen, ``An incremental local algorithm for better voltage
  control in distribution networks,'' \emph{Proc. of IEEE Conference on
  Decision and Control (CDC)}, pp. 2396--2402, 2016.

\bibitem{jahangiri2013distributed}
P.~Jahangiri and D.~C. Aliprantis, ``Distributed {Volt/VAr} control by {PV}
  inverters,'' \emph{IEEE Trans. on power systems}, vol.~28, no.~3, pp.
  3429--3439, 2013.

\bibitem{robbins2013two}
B.~A. Robbins, C.~N. Hadjicostis, and A.~D. Dom{\'\i}nguez-Garc{\'\i}a, ``A
  two-stage distributed architecture for voltage control in power distribution
  systems,'' \emph{IEEE Trans. on Power Systems}, vol.~28, no.~2, pp.
  1470--1482, 2013.

\bibitem{turitsyn2011options}
K.~Turitsyn, P.~\v{S}ulc, S.~Backhaus, and M.~Chertkov, ``Options for control
  of reactive power by distributed photovoltaic generators,'' \emph{Proc. of
  the IEEE}, vol.~99, no.~6, pp. 1063--1073, 2011.

\bibitem{andren2015stability}
F.~Andr{\'e}n, B.~Bletterie, S.~Kadam, P.~Kotsampopoulos, and C.~Bucher, ``On
  the stability of local voltage control in distribution networks with a high
  penetration of inverter-based generation,'' \emph{IEEE Trans. on Industrial
  Electronics}, vol.~62, no.~4, pp. 2519--2529, 2015.

\bibitem{zhang2013local}
B.~Zhang, A.~D. Dominguez-Garcia, and D.~Tse, ``A local control approach to
  voltage regulation in distribution networks,'' \emph{Proc. of IEEE North
  American Power Symposium (NAPS), 2013}, pp. 1--6, 2013.

\bibitem{chen2016reverse}
L.~Chen and S.~You, ``Reverse and forward engineering of frequency control in
  power networks,'' \emph{IEEE Trans. on Automatic Control}, vol.~62, no.~9,
  pp. 4631--4638, 2017.

\bibitem{li2016connecting}
N.~Li, C.~Zhao, and L.~Chen, ``Connecting automatic generation control and
  economic dispatch from an optimization view,'' \emph{IEEE Trans. on Control
  of Network Systems}, vol.~3, no.~3, pp. 254--264, 2016.

\bibitem{Zhao-2012-LC-SGC}
C.~Zhao, U.~Topcu, and S.~Low, ``Swing dynamics as primal-dual algorithm for
  optimal load control,'' in \emph{Proceedings of IEEE SmartGridComm}, 2012.

\bibitem{Zhang-2013-ACC}
X.~Zhang and A.~Papachristodoulou, ``A real-time control framework for smart
  power networks with star topology,'' \emph{Proceedings of American Control
  Conference}, 2013.

\bibitem{Zhao-2014-TAC}
C.~Zhao, U.~Topcu, N.~Li, and S.~Low, ``Design and stability of load-side
  primary frequency control in power systems,'' \emph{IEEE Transactions on
  Automatic Control}, vol.~59, pp. 1177--1189, 2014.

\bibitem{Mallada-2014-IFAC}
E.~Mallada and S.~H. Low, ``Distributed frequency-preserving optimal load
  control,'' \emph{Proceedings of the 19th IFAC World Congress}, 2014.

\bibitem{Dorfler-2014-CDC}
F.~Dorfler, J.~W. Simpson-Porco, and F.~Bullo, ``Breaking the hierarchy:
  Distributed control {\&} economic optimality in microgrids,'' \emph{IEEE
  Transactions on Control of Network Systems}, vol.~3, no.~3, pp. 241--253,
  2016.

\bibitem{zhou2015new}
X.~Zhou and L.~Chen, ``A new perspective to synchronization in networks of
  coupled oscillators: Reverse engineering and convex relaxation,''
  \emph{IFAC-PapersOnLine}, vol.~48, no.~22, pp. 40--45, 2015.

\bibitem{baran1989network}
M.~E. Baran and F.~F. Wu, ``Network reconfiguration in distribution systems for
  loss reduction and load balancing,'' \emph{IEEE Trans. on Power Delivery},
  vol.~4, no.~2, pp. 1401--1407, 1989.

\bibitem{liu2017signal}
Z.~Liu, J.~Shihadeh, S.~You, G.~Ding, X.~Zhou, and L.~Chen,
  ``Signal-anticipating in local voltage control in distribution systems,''
  \emph{arXiv}, 2018.

\bibitem{BoVa04}
S.~Boyd and L.~Vandenberghe, \emph{Convex Optimization}.\hskip 1em plus 0.5em
  minus 0.4em\relax Cambridge University Press, 2004.

\bibitem{wen2004optimal}
J.~Wen, Q.~Wu, D.~Turner, S.~Cheng, and J.~Fitch, ``Optimal coordinated voltage
  control for power system voltage stability,'' \emph{IEEE Trans. on Power
  Systems}, vol.~19, no.~2, pp. 1115--1122, 2004.

\bibitem{varga2009matrix}
R.~S. Varga, \emph{Matrix iterative analysis}.\hskip 1em plus 0.5em minus
  0.4em\relax Springer Science \& Business Media, 2009, vol.~27.

\bibitem{galor2007discrete}
O.~Galor, \emph{Discrete dynamical systems}.\hskip 1em plus 0.5em minus
  0.4em\relax Springer Science \& Business Media, 2007.

\bibitem{zimmerman2011matpower}
R.~D. Zimmerman, C.~E. Murillo-S{\'a}nchez, and R.~J. Thomas, ``{MATPOWER}:
  {S}teady-state operations, planning, and analysis tools for power systems
  research and education,'' \emph{IEEE Trans. on Power Systems}, vol.~26,
  no.~1, pp. 12--19, 2011.
  
  \bibitem{farivar2013equilibrium}
M.~Farivar, L.~Chen, and S.~Low, ``Equilibrium and dynamics of local voltage
  control in distribution systems,'' \emph{Proc. of IEEE Conference on Decision
  and Control (CDC)}, pp. 4329--4334, 2013.

\bibitem{farivar2015local}
M.~Farivar, X.~Zhou, and L.~Chen, ``Local voltage control in distribution
  systems: An incremental control algorithm,'' \emph{Proc. of IEEE
  International Conference on Smart Grid Communications (SmartGridComm)}, pp.
  732--737, 2015.

\bibitem{zhou2015pseudo}
X.~Zhou, M.~Farivar, and L.~Chen, ``Pseudo-gradient based local voltage control
  in distribution networks,'' \emph{Proc. of IEEE Annual Allerton Conference on
  Communication, Control, and Computing}, pp. 173--180, 2015.

\end{thebibliography}

\end{document}